\documentclass[10pt]{article}

\usepackage{amsmath}
\usepackage{amsthm}
\usepackage{mathrsfs}	
\usepackage{amssymb}
\usepackage{amscd}
\usepackage{setspace}	
\usepackage{tikz}		
\usepackage{tikz-cd}
\usetikzlibrary{matrix}	
\usepackage{enumerate}	
\usepackage{bbm}		

\usepackage[utf8]{inputenc}	
\usepackage[english]{babel}	
\usepackage{blindtext}		
\usepackage[affil-it]{authblk}	

\usepackage{faktor}			
\usepackage{xparse}			
\DeclareDocumentCommand{\rfaktor}{m O{-0.5} m O{0.5}}{
  	\raisebox{#2\height}{\ensuremath{#1}}
 	 \mkern-5mu\diagdown\mkern-4mu
 	 \raisebox{#4\height}{\ensuremath{#3}}
	}
\usepackage{cancel} 
\usepackage{indentfirst}

\DeclareMathOperator{\f.a.}{~for~all~}

\newcommand{\ds}{\displaystyle}
\newcommand{\abs}[1]{\left\lvert#1\right\rvert}   					
\newcommand{\norm}[1]{\left\lVert#1\right\rVert}					
\newcommand{\iso}{\cong}									
\newcommand{\lrangle}[1]{\left\langle#1\right\rangle}					
\newcommand{\opint}[1]{\left(#1\right)}							
\newcommand{\set}[1]{\left\{#1\right\}}		 					
\newcommand{\maxset}[1]{\text{max}\left\{#1\right\}}					
\newcommand{\minset}[1]{\text{min}\left\{#1\right\}}					
\newcommand{\cu}[1]{C_u^*\left(#1\right)}							
\newcommand{\cualg}[2]{\mathbb{C}_u^{#1}\left[#2\right]} 			
\newcommand{\sub}{\subseteq}

\newcommand{\h}{\hspace{.1 in}}								
\newcommand{\cl}[1]{\overline{#1}}								
\newcommand{\os}[2]{\overset{#1}{#2}}							

\DeclareMathOperator{\dd}{d}

\DeclareMathOperator{\im}{im}
\DeclareMathOperator{\ev}{ev}

\newcommand{\N}{\mathbb{N}}			

\newcommand{\C}{\mathbb{C}}

\newcommand{\D}{\mathbb{D}}

\newcommand{\BB}{\mathscr{B}} 	

\newcommand{\KK}{\mathscr{K}}
\newcommand{\LL}{\mathscr{L}}

\newcommand{\A}{\mathcal{A}} 	
\newcommand{\HH}{\mathcal{H}}

\newcommand{\UU}{\mathcal{U}} 
\newcommand{\VV}{\mathcal{V}}
\newcommand{\WW}{\mathcal{W}}

\newtheorem{theorem}{Theorem}[section]     

\newtheorem{lemma}[theorem]{Lemma}

\theoremstyle{remark}
\newtheorem{remark}[theorem]{Remark}

\theoremstyle{definition}
\newtheorem{definition}[theorem]{Definition}

\usepackage[utf8]{inputenc}

\title{The Hochschild Cohomology of Uniform Roe Algebras}
\author{Matthew Lorentz}
\date{}

\begin{document}

\maketitle

\begin{abstract}

In Rufus Willett's and the authors paper ``Bounded Derivations on Uniform Roe Algebras" \cite{lorentz2020} we showed that all bounded derivations on a uniform Roe algebra $C^*_u(X)$ associated to a bounded geometry metric space $X$ are inner.
This naturally leads to the question of whether or not the higher dimensional Hochschild cohomology groups of the uniform Roe algebra vanish also.
While we cannot answer this question completely, we are able to give necessary and sufficient conditions for the vanishing of $H^n_c(C_u^*(X),C_u^*(X))$.\\
Lastly, we show that if the norm continuous Hochschild cohomology of a uniform Roe algebra vanishes in all dimensions then the ultraweak-weak* continuous Hochschild cohomology of that uniform Roe algebra vanishes also.
\end{abstract}

\section{\centering Introduction}

	Uniform Roe algebras are a well-studied class of non-separable $C^*$-algebras associated to metric spaces.  They were originally introduced for index-theoretic purposes, but are now studied for their own sake as a bridge between $C^*$-algebra theory and coarse geometry, as well as having interesting applications to single operator theory and mathematical physics.  Due to the presence of $\ell^\infty(X)$ as a diagonal maximal abelian subalgebra, they have a somewhat von Neumann algebraic feel, but are von Neumann algebras only in the trivial finite-dimensional case.   
Moreover, in many ways they are quite tractable as $C^*$-algebras, often having good regularity properties such as nuclearity.

Hochschild cohomology was introduced by Gerhard Hochschild in his 1945 paper \textit{On the Cohomology Groups of an Associative Algebra} \cite{hochschild1945cohomology}. The Hochschild cohomology of associative algebras has become a useful object of study in many fields of mathematics such as representation theory, mathematical physics, and noncommutative geometry, to name a few. 

In Rufus Willett's and the authors paper ``Bounded Derivations on Uniform Roe Algebras" \cite{lorentz2020} we showed that all bounded derivations on a uniform Roe algebra $C^*_u(X)$ associated to a bounded geometry metric space $X$ are inner.
That all bounded derivations are inner is equivalent to the first norm continuous Hochschild cohomology group $H^1_c(C_u^*(X),C_u^*(X))$ vanishing.
Indeed, the Hochschild coboundary operator from a C*-algebra $\A$ to the linear maps from $\A$ to itself is given by 
$$
\partial a (b) = ab-ba,~ a,b \in \A.
$$
Thus, $\partial a$ is an inner derivation.
Next, the coboundary operator from a linear map $\phi$ to bilinear map from $\A$ to itself is given by
$$
\partial\phi(a,b) = a\phi(b) - \phi(ab) + \phi(a)b.
$$
Hence, the kernel of this coboundary operator is the set of derivations on $\A$. 
So, taking this kernel and modding out by the image of the previous coboundary, if zero, means that all derivations on $\A$ are inner.
Thus,  the first Hochschild cohomology of uniform Roe algebras associated to bounded geometry metric spaces vanishes.
It is then natural to ask if the higher groups $H^n_c(C_u^*(X),C_u^*(X))$ also vanish.

	The question of whether or not the Hochschild cohomology vanishes in all dimensions for a hyperfinite von Neumann algebra has been answered completely by Kadison and Ringrose.

	\begin{theorem}[\cite{MR318907} Theorem 3.1]
	\label{iszero}
	The Hochschild cohomology of a hyperfinite von Neumann algebra vanishes in all dimensions. 
	\end{theorem}
	
Additionally, there have been many advancements for von Neumann algebras in general. For examples see Sinclair and Smith's book ``Hochschild cohomology of von Neumann algebras" \cite{sinclair1995hochschild}.

	 While we are not able to answer the question of whether or not the Hochschild cohomology vanishes in all dimensions for uniform Roe algebras, in Section \ref{sec5} we are able to give conditions for the vanishing of the higher dimensional Hochschild cohomology of a uniform Roe algebra. Specifically:
	 
	 \begin{theorem}[cf. Theorem \ref{main 2}] If every element of $H_c^n(\cu{X})$ admits a weakly continuous representation, then $H_c^n(\cu{X})=0$. 
	\end{theorem} 
	
\noindent	Note that, since all derivations are automatically weakly continuous by \cite{10.2307/1970433} Lemma 3, the previous theorem contains the derivations theorem as a special case.
	
	Lastly, we show:
	\begin{theorem}[cf. Theorem \ref{main 3}]
	If the norm continuous Hochschild cohomology of a uniform Roe algebra vanish in all dimensions then the ultraweak-weak* continuous Hochschild cohomology of that uniform Roe algebra vanish in all dimensions.
	\end{theorem}

	The paper is organized as follows. In section \ref{sec2} we define uniform Roe algebras and introduce some of their properties.

	Next, we review a technique to average over amenable groups. While most of this method seems well known, it is essential to the proofs that follow. Thus, we construct it in section \ref{sec3}. We then state a key technical result from Braga and Farah which we `upgrade' to multilinear maps so that it may be applied to Hochschild cohomology. 

	Section \ref{hochcoho} will begin with the definition of the Hochschild complex and Hochschild cohomology as they apply to multilinear maps from a C*-algebra $\A$ to a Banach $\A$-bimodule $\mathcal{V}$. We then review many properties of these cohomologies from Sinclair and Smith's book, \textit{Hochschild cohomology of von Neumann algebras} \cite{sinclair1995hochschild}.

	Lastly, in Section \ref{sec6}, we review the connection between the Hochschild cohomology of ultraweak-weak* continuous multilinear maps and the Hochschild cohomology of norm continuous multilinear maps. We then conclude by showing that if the norm continuous Hochschild cohomology of uniform Roe algebras vanishes in all dimensions then so does the ultraweak-weak* continuous Hochschild cohomology.
	
\subsection*{Acknowledgements}
	
	The author would like to thank several academics for assisting my exploration of Hochschild cohomology: Professor Stuart White who first suggested that I look into Hochschild cohomology while visiting the University of Hawai`i, Professor Roger Smith who assisted in the use of his theory through email, and Professor Ilijas Farah for several useful conversations at the Young Mathematicians in C*-Algebras (YMC*A) conference in Copenhagen 2019. I would also like to thank my advisor Rufus Willett for numerous suggestions and advice. 
	My research was partially funded by the National Science Foundation grants DSM-1564281 and DSM-1901522, for which I am grateful.

\section{\centering Preliminaries}
\label{sec2}
	Inner products are linear in the first variable. For a Hilbert space $\HH$ we denote the space of bounded operators on $\HH$ by $\BB(\HH)$, and the space of compact operators by $\KK(\HH)$.   

	The Hilbert space of square-summable sequences on a set $X$ is denoted $\ell^2(X)$, and the canonical basis of $\ell^2(X)$ will be denoted $(\delta_x)_{x\in X}$.  For $a\in \BB(\ell^2(X))$ we define its matrix entries by
	$$
	a_{xy}:= \lrangle{\delta_x,a\delta_y}.
	$$
\subsection{Uniform Roe Algebras}

	We now give some basic definitions regarding uniform Roe algebras.

	\begin{definition}[propagation, uniform Roe algebra]
	Let $X$ be a metric space and $r\geq 0$.  An operator $a\in \BB(\ell^2(X))$ \emph{has propagation at most $r$} if $a_{xy}=0$ whenever 
	$d(x,y) > r$ for all $(x,y)\in X\times X$.  In this case, we write prop$(a) \leq r$. 
	The set of all operators with propagation at most $r$ is denoted $\cualg{r}{X}$.  We define
	$$
	\cualg{}{X} := \{a\in \BB(\ell^2(X)) : \text{prop}(a) < \infty\};
	$$
	it is not difficult to see that this is a $*$-algebra.  The \textit{uniform Roe algebra}, denoted $C^*_u(X)$, is defined to be the norm closure of $\C_u[X]$ under the norm inherited from $\BB(\ell^2(X))$.
	\end{definition}

	\begin{definition}[$\epsilon$-$r$-\textit{approximated}]
	\label{er-approx}
	Let $X$ be a metric space.
	Given $\epsilon >0$ and $r >0$, 
	an operator $a\in \BB(\ell^2(X))$ can be $\epsilon$-$r$-\textit{approximated} if there exists 
	$b \in \cualg{r}{X}$ such that $\norm{a-b} \leq \epsilon$.
	Note that an operator $a\in\BB(\ell^2(X))$ is in the uniform Roe algebra if and only if for all $\epsilon>0$ there exists an $r$ such that $a$ can be $\epsilon$-$r$-approximated.
	\end{definition}
	
We will be exclusively interested in uniform Roe algebras associated to bounded geometry metric spaces as in the next definition.
	
	\begin{definition}[bounded geometry]\label{bg}
	A metric space $X$ is said to have \textit{bounded geometry} if for every $r\geq 0$ there exists an $N_r \in \N$ 
	such that for all $x\in X$, the ball of radius $r$ about $x$ has at most $N_r$ elements.
	\end{definition}

	\section{Averaging over Amenable Groups} 
	\label{sec3}

In this section, we summarize some facts we need about averaging operators over an amenable group. 
	
	Let $G$ be a discrete (possibly uncountable) group.  If $A$ is a complex Banach space, we let $\ell^\infty(G,A)$ denote the Banach space of bounded functions from $G$ to $A$ equipped with the supremum norm; in the case $A=\C$, we just write $\ell^\infty(G)$.  We also equip $\ell^\infty(G,A)$ with the right-action of $G$ defined for $a\in \ell^\infty(G,A)$ and $h,g\in G$ by
$$
(ag)(h):=a(hg^{-1}).
$$
If $Z$ is any set, a function $\phi:\ell^\infty(G,A)\to Z$ is \emph{invariant} if $\phi(ag)=\phi(a)$ for all $a\in \ell^\infty(G,A)$ and $g\in G$.  

Recall that $G$ is amenable if there exists an \emph{invariant mean} on $\ell^\infty(G)$, i.e.\ an invariant function $\Phi:\ell^\infty(G)\to \C$ that is also a state. 
Fix an invariant mean $\Phi$ on $\ell^\infty(G)$ and let $B$ be a complex Banach space with dual $B^*$.  We may upgrade an invariant mean on $\ell^\infty(G)$ to an invariant contractive linear map $\ell^\infty(G,B^*)\to B^*$ in the following way. Let $b\in B$, $g\in G$, and $a\in \ell^\infty(G,B^*)$, and write $\langle b,a(g)\rangle$ for the pairing between $b$ and $a(g)$.  
Define a map
$$
\Psi_{b,a}:G \to \C \h \text{by} \h g \mapsto \lrangle{b,a(g)}.
$$
Note that $\abs{\Psi_{b,a}(g)} = \abs{\lrangle{b,a(g)}}\leq \norm{a}_{\ell^{\infty}(G,B^*)}\norm{b}_B$ for all $g\in G$.
Hence, $\Psi_{b,a} \in \ell^{\infty}(G)$ for all $b\in B$ and for all $a\in \ell^{\infty}(G,B^*)$
so that when we  apply $\Phi$ we get a complex number $\Phi(\Psi_{b,a})$.
We now define a map
$$
\Phi_a:B\to \C \h \text{by} \h b \mapsto \Phi(\Psi_{b,a}).
$$
Observe that, since $\Phi$ is a state, 
\begin{equation}
\label{Psicontraction}
\abs{\Phi_a(b)} = \abs{\Phi(\Psi_{b,a})} \leq \norm{\Psi_{b,a}}_{\ell^{\infty}(G)} \leq \norm{a}_{\ell^{\infty}(G,B^*)}\norm{b}_B
\end{equation}
and so $\Phi_a \in B^*$. 
Lastly, we define
$$
\Psi:\ell^{\infty}(G,B^*) \to B^* \h \text{by} \h a \mapsto \Phi_a.
$$

The proof of the next lemma is straightforward and so we leave it for the reader.

\begin{lemma}
\label{avg lem}
With notation as above, the map
$$
\Psi:\ell^{\infty}(G,B^*) \to B^* \h \text{defined by} \h a \mapsto \Phi_a
$$
is uniquely determined by the condition 
\begin{equation}\label{avg}
\lrangle{b,\Psi(a)} = \Phi(\lrangle{b,a(\cdot)})
\end{equation}
for $b\in B$ and $a\in \ell^\infty(G,B^*)$.  It is contractive, linear, invariant, and acts as the identity on constant functions. 

\qed
\end{lemma}

Before we conclude with the properties of $\Psi$ we will introduce an action by a C*-algebra $\A$ on $B^*$.
We then `upgrade' this action to an action on $\ell^{\infty}(G,B^*)$ and $B$.
Once this is done we will be able to show that $\Psi$ behaves `like' a conditional expectation.
That is, for $x, y \in \A,~ f \in \ell^{\infty}(G,B^*)$, $\Psi(x\cdot f \cdot y) = x \cdot \Psi(f) \cdot y$.
First, we will need a few definitions and lemmas. Since the next two definitions will be used elsewhere we temporarily change our notation; that is, $B^* = \mathcal{V}$. 

	\begin{definition}[Banach $\mathcal{A}$-bimodule]
	\label{Banach bimodule}
	Let $\mathcal{A}$ be a C*-algebra. We say that $\mathcal{V}$ is a \textit{Banach $\mathcal{A}$-bimodule} if $\mathcal{A}$ acts nondegenerately on $\mathcal{V}$ from both the left and the right and $\mathcal{V}$ has a norm under which it is a Banach space.
	Moreover, the norm on $\mathcal{V}$ satisfies
	$$
	\norm{av}_{\mathcal{V}} \leq \norm{a}_{\mathcal{A}}\norm{v}_{\mathcal{V}} \h \text{and} \h \norm{va}_{\mathcal{V}} \leq \norm{v}_{\mathcal{V}}\norm{a}_{\mathcal{A}} \h \text{for all} \h a\in\mathcal{A},~ v\in\mathcal{V}.
	$$
	\end{definition}
	
	\begin{definition}[Dual module]
	\label{dual module}
	Let $\mathcal{A}$ be a C*-algebra.
	we say that $\mathcal{V}$ is a \textit{dual module} over $\mathcal{A}$ if:
		
	\begin{enumerate}[(i)]
	\item
	$\mathcal{V}$ is a Banach $\mathcal{A}$-bimodule,
	\item
	$\mathcal{V}$ has a pre-dual $\mathcal{V}_{*}$, 
	\item
	and for $x\in \mathcal{V}$ the maps 
	$$
	L_a: x\mapsto a\cdot x \h \text{and} \h R_a: x\mapsto x\cdot a
	$$
	 are weak* continuous for all $a\in \mathcal{A}$.
	 \end{enumerate}
	 \end{definition}
	
	\begin{lemma}
	Let $\A$ be a C*-algebra and suppose that $B^*$ is a dual $\A$-bimodule.
	Then we can make $\ell^{\infty}(G, B^*) $ a Banach $\A$-bimodule via
	$$
	(x\cdot f)(g) := x\cdot f(g), \h \text{and} \h (f\cdot x)(g) := f(g) \cdot x
	$$
	where $f \in \ell^{\infty}(G,B^*),~x\in \A$, and $g \in G$.
	\qed
	\end{lemma}
	
	We now upgrade the action of $\A$ on $B^*$, to an action on $B$.
	
	\begin{lemma}
	Let $\A$ be a C*-algebra and
	suppose that $B^*$ is a dual $\A$-bimodule.
	Then we can make $B$ an $\A$-bimodule via actions that satisfy 
	$$
	\lrangle{a \cdot b , b^*} = \lrangle{b, a^*\cdot b^*} \h \text{and}\h \lrangle{b\cdot a, b^*} = \lrangle{b, b^*\cdot a^*} \h\text{where} \h b \in B.
	$$
	\end{lemma}
	
	\begin{proof}
	First, we dualize $B^*$ with respect to the $\sigma(B^*,B)$ topology which we denote by $B^{*\dagger}$.
	Note that the topology on $B^{*\dagger}$ is the weakest topology that makes the evaluation maps $ \ev_b :b^* \to \C$ continuous.
	Moreover, by \cite{reed2012methods} Theorem IV.20, $B^{*\dagger} \iso B$.
	Thus, dualizing the maps $L_a$ and $R_a$ with respect to the $\sigma(B^*,B)$ topology the maps $L_a^{\dagger}$ and $R_a^{\dagger}$ are maps on $B$ for all $a\in \A$.
	\end{proof}

\begin{lemma}
\label{ce prop}
Let $\A$ be a C*-algebra and
suppose that $B^*$ is a dual $\A$-bimodule. 
Then the averaging operator $\Psi: \ell^{\infty}(G,B^*) \to B^*$ as defined above has the property that
$$
\Psi(a \cdot f) =a \cdot \Psi(f) \h \text{and} \h \Psi(f \cdot a) = \Psi(f) \cdot a
$$
\end{lemma}

\begin{proof}
Let $b\in B,~f \in \ell^{\infty}(G,B^*)$, and $a \in \A$.
Observe that
$$
\lrangle{b, a\cdot \Psi(f)} =  \lrangle{a^*\cdot b, \Psi(f)}  =\Phi(\lrangle{a^*\cdot b,  f(\cdot)}) = \Phi(\lrangle{b, (a\cdot f)(\cdot)}) = \lrangle{b, \Psi(a\cdot f)}
$$
with a similar calculation when $\A$ acts on the right.
\end{proof}

We will be using this machinery to average over multilinear maps. Rather then defining new maps for each situation, and since our averaging operator enjoys all of the properties (except for countable additivity) as if we were integrating over a normalized Haar measure, we will use integral notation to denote our averaging operator.
That is, if $\Psi$ is as above for $f\in \ell^{\infty}(G,B^*)$ and $g \in G$ we define 
$$
\Psi(f) =: \int_Gf(g) \dd\mu(g).
$$
Note that, in the non-compact amenable case, $\mu$ is not a measure; it serves only as a notational device.

We will apply this machinery in the case that $B=\mathcal{L}^1(\ell^2(X))$ is the trace class operators on $\ell^2(X)$.  In this case, the dual $B^*$ canonically identifies with $\BB(\ell^2(X))$: indeed, if $\text{Tr}$ is the canonical trace on $\mathcal{L}^1(\ell^2(X))$, $b\in \mathcal{L}^1(\ell^2(X))$, and $a\in \BB(\ell^2(X))$, then the pairing inducing this duality isomorphism is defined by
\begin{equation}\label{pair form}
\langle b,a\rangle:=\text{Tr}(ba).
\end{equation} 

The next lemma says that our averaging process behaves well with respect to propagation.  The main point of the lemma is that the collection of operators in $\BB(\ell^2(X))$ that have propagation at most $r$ is weak-$*$ closed for the weak-$*$ topology inherited from the pairing with $\mathcal{L}^1(\ell^2(X))$.

\begin{lemma}\label{avg prop lem}
With notation as above, if $r\geq 0$ and $f\in \ell^\infty(G,\BB(\ell^2(X)))$ is such that the propagation of each $f(g)$ is at most $r$ for all $g\in G$, then the propagation of $\int_{G}f(g)\dd\mu(g)$ is also at most $r$.
\end{lemma}

\begin{proof}
Let $e_{xy}\in \mathcal{L}^1(\ell^2(X))$ be the standard matrix unit.  Then one computes using line \eqref{pair form} above that for any $a\in \BB(\ell^2(X))$,
\begin{equation}\label{munit}
\langle e_{yx},a\rangle=\text{Tr}(e_{yx}a)=a_{xy}.
\end{equation}
Using lines \eqref{avg} and \eqref{munit}, we see that
$$
\Big\langle e_{yx},\int_{G}f(g)\dd\mu(g)\Big\rangle= \int_{G}\langle e_{yx}, f(g)\rangle\dd\mu(g)=\int_{G}f(g)_{xy}\dd\mu(g),
$$
where the last expression means the image of the function 
$$
G\to \C, \quad g\mapsto f(g)_{xy} 
$$
under the invariant mean.
Now, if $d(x,y)>r$, we have that $f(g)_{xy}=0$ for all $g\in G$, and therefore that $\int_{G}f(g)_{xy}\dd\mu(g)=0$.  Hence, by the above computation, 
$$
d(x,y)>r \quad \text{implies} \quad \Big\langle e_{yx},\int_{G}f(g)\dd\mu(g)\Big\rangle=0.
$$
Using line \eqref{munit}, this says that $\int_{G}f(g)\dd\mu(g)$ has propagation at most $r$, so we are done.
\end{proof}
	
	
\subsection{A Result of Braga and Farah}
Note that in the averaging process from the previous subsection, convergence is happening in the weak-$*$ topology of $\BB(\HH)$. 
However, by Lemma \ref{avg prop lem}, we know that the averaging process behaves well with uniformly finite propagation operators. 
In this subsection, we present a result of Braga and Farah from \cite[Lemma 4.9]{braga2018rigidity} (see Theorem \ref{bf main} below) which will allow us to work with uniformly finite propagation operators.  
This theorem will allow us to uniformly $\epsilon$-$r$-approximate (Definition \ref{er-approx}) $f\in \ell^\infty(\mathcal{U},\BB(\ell^2(X)))$ where $\mathcal{U}$ is the unitary group of $\ell^{\infty}(X)$.
That is,  given $\epsilon>0$, there exists a \textit{single} $r>0$ such that for all $u\in \mathcal{U}$ such that $f(u)\in\BB(\ell^2(X))$ can be $\epsilon$-$r$-approximated.

To state the result, let $\mathbb{D}:=\{z\in \C\mid |z|\leq 1\}$ denote the closed unit disk in the complex plane.  Let $I$ be a countably infinite set, and let $\mathbb{D}^I$ denote as usual the space of all $I$-indexed tuples $\lambda:=(\lambda_i)_{i\in I}$ with each $\lambda_i\in \mathbb{D}$.  We fix this notation throughout this section.

\begin{definition}[symmetrically summable]\label{sss}
A sequence $(a_i)_{i\in I}$ is \emph{symmetrically summable} if for all $\lambda\in \mathbb{D}^I$, the sum $\sum_{i\in I} \lambda_i a_i$ converges in the weak operator topology
to an element of $C^*_u(X)$.  If $(a_i)$ is symmetrically summable and $\lambda=(\lambda_i)$ is in $\mathbb{D}^I$, we write $a_\lambda$ for the operator $\sum_{i\in I} \lambda_i a_i$.   
\end{definition}

\begin{theorem}[Lemma 4.9 \cite{braga2018rigidity}]
\label{bf main}
Let $(a_i)$ be a symmetrically summable collection of operators in $C^*_u(X)$.  Then for any $\epsilon>0$ there exists $r>0$ such that for all $\lambda \in \mathbb{D}^I$, the operator $a_\lambda$ is $\epsilon$-$r$-approximated.
\end{theorem}

The content of the result is the order of quantifiers: the point is that given an $\epsilon>0$ there is an $r>0$ that works for all the $a_\lambda$ at once.  The proof of Theorem \ref{bf main} proceeds via an application of the Baire category theorem to the following sets.  

\begin{definition}\label{ureps}
Say $(a_i)$ is symmetrically summable, and for any $\epsilon,r>0$ define
$$
U_{\epsilon,r}:=\{\lambda\in \mathbb{D}^I \mid a_{\lambda} \text{ can be $\epsilon$-$r$-approximated}\}.
$$
\end{definition}

Note that the hypothesis of Theorem \ref{bf main} says that for any $\epsilon>0$, 
\begin{equation}\label{d union}
\mathbb{D}^I=\bigcup_{r=1}^\infty U_{\epsilon,r},
\end{equation}
while the conclusion of Theorem \ref{bf main} says that for any $\epsilon>0$ there exists $r$ such that $\mathbb{D}^I=U_{\epsilon,r}$.

\subsection{A Generalization of Braga and Farah's Lemma,\\ Multilinear Version} 

	\begin{definition}[separately symmetrically summable]
	For a finite sequence of countable index sets $\set{I_n}_{n=1}^N,~N<\infty$, a uniformly bounded family of operators $(a_{(i_1, \dots, i_N)})_{(\cl{i}\in \prod_{n=1}^N I_n)}\subseteq \cu{X}$ is $N$ \textit{separately symmetrically summable} if the following condition holds.
	
	For any $(1 \leq k \leq N)$, and for each fixed 
	$$
	\set{\lambda^{(1)},\dots, \lambda^{(k-1)}, \lambda^{(k+1)},\dots, \lambda^{(N)}} \in \prod^N_{\substack{n=1\\n\neq k}}\D^{I_n}
	$$ the sum
	$$
	\sum_{i_{k} \in I_{k}}\lambda^{(k)}_{i_{k}} a_{(\lambda^{(1)}, \dots, \lambda^{(k-1)},i_{k}, \lambda^{(k+1)}, \dots, \lambda^{(N)})}
	$$
	converges in the weak operator topology to an element
	$$
	a_{(\lambda^{(1)}, \dots,\lambda^{(k)}, \dots, \lambda^{(N)})} \in \cu{X}.
	$$
	Additionally, 
	$$
	\f.a. \set{\lambda^{(1)},\dots,  \lambda^{(N)}} \in \prod^N_{n=1}\D^{I_n},
	\h\h
	\sup_{(\lambda^{(1)}, \dots, \lambda^{(N)})}\norm{a_{(\lambda^{(1)}, \dots, \lambda^{(N)})}} < \infty.
	$$

	Note that, if $(a_{(i_1, \dots, i_{N+1})})_{(\cl{i}\in \prod_{n=1}^{N+1} I_n)}$ is $(N+1)$ separately symmetrically summable, then for any fixed $\eta \in \D^{I_{N+1}}$, $(a_{(i_1, \dots, i_{N},\eta)})_{(\cl{i}\in \prod_{n=1}^N I_n)}$  is N separately symmetrically summable.
	\end{definition}
	
	We are now ready to generalize Braga and Farah's Lemma.
	
	\begin{theorem}
	\label{bf corollary}
	
	Suppose that $$(a_{(i_1, \dots, i_N)})_{(\cl{i}\in \prod_{n=1}^{N} I_n)}\subseteq \cu{X}$$ is $N$ separately symmetrically summable.
	Then for any $\epsilon >0$ there exists an $r >0$ such that for all $(\lambda^{(1)}, \dots, \lambda^{(N)})  \in \prod_{n=1}^N\D^{I_n}$, the operator
	$
	a_{(\lambda^{(1)}, \dots, \lambda^{(N)})}
	$
	is $\epsilon$-$r$-approximated.
	\end{theorem}

	To prove this theorem we induct on $N$. 
	However, we will need a few lemmas and a definition first.
	Note that the base case is handled by Theorem \ref{bf main}.

\begin{lemma}
\label{bf1}
\begin{enumerate}[(i)]
\item If $a$ is a bounded operator on $\ell^2(X)$ such that for all finite rank projections $p$ in $\ell^\infty(X)$ the product $pap$ can be $\epsilon$-$r$-approximated, then $a$ itself can be $\epsilon$-$r$-approximated. 
\item Say $a$ is a bounded operator on $\ell^2(X)$ and $\epsilon,r>0$ are such that for all $\delta>0$, $a$ can be $(\epsilon+\delta)$-$r$-approximated.  Then $a$ can be $\epsilon$-$r$-approximated.
\end{enumerate}
\end{lemma}

\begin{proof}
\begin{enumerate}[(i)]
\item Let $J$ be the net of all finite rank projections in $\ell^\infty(X)$, equipped with the usual operator ordering.  For each $p\in J$, choose $b_p\in \C_u^r[X]$ such that $\|pap-b_p\|\leq \epsilon$.  Then the net $(b_p)_{p\in J}$ is norm bounded, so has a weak operator topology convergent subnet, say $(b_p)_{p\in J'}$, converging to some bounded operator $b$ on $\ell^2(X)$.  Note moreover that $\lim_{p\in J'}p$ equals the identity in the weak operator topology, and so $\lim_{p\in J'}pap=a$ and $\lim_{p\in J'}(pap-b_p)=a-b$ in the weak operator topology.

\hspace{.15in}Now, as weak operator topology limits do not increase norms, we see that 
$$
\|a-b\|\leq \limsup_{p\in J'}\|pap-b_p\|\leq \epsilon.
$$

Hence to complete the proof, it suffices to show that $b$ is in fact in $\C_u^r[X]$.  Indeed, for each $(x,y)\in X\times X$, the function taking a bounded operator $c$ on $\ell^2(X)$ to its matrix entry $c_{xy}$ is weak operator topology continuous.  
Hence, if $d(x,y)>r$ then
$$
b_{xy}=\lim_{p\in J'}\big((b_p)_{xy}\big)=0 \h \text{and so} \h b\in \cualg{r}{X}.
$$

\item For each $n$, let $b_n\in \C_u^r[X]$ be such that $\|a-b_n\|\leq \epsilon+1/n$.  As in the previous part, there is a subnet $(b_{n_j})_{j\in J}$ of the sequence $(b_n)$ that converges to some $b\in \C_u^r[X]$ in the weak operator topology.  As weak operator topology limits cannot increase norms, we see that 
$$
\|a-b\|\leq \limsup_{j\in J} \|a-b_{n_j}\|\leq \limsup_{j\in J} (\epsilon+1/n_j)=\epsilon,
$$  
which shows that $a$ can be $\epsilon$-$r$-approximated as claimed. \qedhere
\end{enumerate}

\end{proof}

\begin{lemma}\label{bf2}
Say $(x_i)_{i\in I}$ is a collection in a Banach space such that $\sum_i \lambda_i x_i$ converges in norm for all $(\lambda_i)\in \mathbb{D}^I$.  Then for any $\delta>0$ there exists a finite subset $F$ of $I$ such that for all $(\lambda_i)\in \mathbb{D}^I$
$$
\Bigg\|\sum_{i\in I\setminus F} \lambda_i x_i\Bigg\|<\delta.
$$ 
\end{lemma}

\begin{proof}
For notational convenience, identify $I$ with $\N$, so we are just dealing with a sequence $(x_n)$.  Assume for contradiction that there exists $\delta>0$ such that for all $N$ there exists $(\lambda_n)\in \mathbb{D}^\N$ such that 
$$
\Bigg\|\sum_{n>N} \lambda_n x_n\Bigg\|\geq \delta.
$$

We will inductively define sequences $(\lambda^{(m)})_{m=1}^\infty$ of points in $\mathbb{D}^\N$ and \\$N_1<M_1<N_2<M_2<\cdots$ of natural numbers such that for all $m$, 
$$
\Bigg\|\sum_{n=N_m+1}^{M_m} \lambda^{(m)}_n x_n\Bigg\|\geq \delta /2.
$$ 
Indeed, let $m=1$, and let $N_1$ and $\lambda^{(1)}$ be such that 
$$
\Bigg\|\sum_{n>N_1} \lambda_n^{(1)} x_n\Bigg\|\geq \delta.
$$
As $\sum_{n>N_1} \lambda_n^{(1)} x_n$ is norm convergent, there exists $M_1>N_1$ such that 
$$
\Bigg\|\sum_{n>M_1} \lambda_n^{(1)} x_n\Bigg\|\leq  \delta/2
$$
(such exists by our convergence assumption).  Now, having chosen $N_1<M_1<N_2<\cdots<M_m$, let us choose $N_{m+1}>M_{m}$ and $(\lambda)^{(m+1)}$ so that 
$$
\Bigg\|\sum_{n>N_{m+1}} \lambda_n^{(m+1)} x_n\Bigg\|\geq \delta,
$$
and choose $M_{m+1}>N_{m+1}$ such that 
$$
\Bigg\|\sum_{n>M_{m+1}} \lambda_n^{(m+1)} x_n\Bigg\|\leq  \delta/2.
$$
Then the constructed sequences have the desired properties.

Now, define a new sequence $\lambda\in \mathbb{D}^\N$ by the formula 
$$
\lambda_n:=\left\{\begin{array}{ll} \lambda^{(m)}_n, & N_m<n\leq M_m \\ 0, & \text{otherwise}. \end{array}\right.
$$
Then $\sum_{n=1}^\infty \lambda_n x_n$ converges in norm.  In particular, it is Cauchy.  This implies that for all suitably large $m$, $\|\sum_{n=N_m+1}^{M_m} \lambda_n x_n \|<\delta/2$, which contradicts the properties of our construction. 
\end{proof}
	
	\begin{definition}
	\label{N2 r sets}
	Suppose that $(a_{(i_1, \dots, i_N)})_{(\cl{i}\in \prod_{n=1}^{N} I_n)}\subseteq \cu{X}$ is $N$ separately symmetrically summable.
	Let $\lambda = (\lambda^{(1)}, \dots, \lambda^{(N-1)})$. Then for $\eta \in \D^{I_N}$ we let 
	$$
	a_{\lambda, \eta} = \sum_{i_N \in I_N} \eta_{i_N} a_{\lambda, i_N}
	$$
	Then for $\epsilon,r>0$
	define
	$$
	U_{\epsilon,r} := \set{\eta \in \D^{I_N} \mid a_{\lambda,\eta} \text{ is } \epsilon\text{-}r\text{-approximated for all }\lambda\in\prod_{n=1}^{N-1}\D^{I_n}}.
	$$
	\end{definition}
	
	\begin{remark}
	\label{rem ref}
	On the first read it may provide intuition to just consider the $N =2$ case since the proof of the inductive step is only notationally different.
	
	Suppose that $\epsilon>0$ is given.
	If we are considering the $N=2$ case, and $\set{a_{i,j}}_{i\in I, j\in J}$ is 2 separately symmetrically summable. 
	Then, for each fixed $\eta \in \D^J$, $\set{a_{i,\eta}}_{i\in I}$ is symmetrically summable so by Theorem \ref{bf main} we may write $\D^J$ as the union in line (\ref{N2 union}).
	
	For the inductive step, suppose that $(a_{(i_1, \dots, i_N)})_{(\cl{i}\in \prod_{n=1}^{N} I_n)}$ is $N$ separately symmetrically summable.
	Then, for each fixed $\eta \in \D^{I_N}$ we have that \\
	$(a_{(i_1, \dots, i_{N-1},\eta)})_{(\cl{i}\in \prod_{n=1}^{N-1} I_n)}$ is $(N-1)$ 
	separately symmetrically summable.
	Thus, by inductive hypothesis we may write $\D^{I_N}$ as the union
	
	\begin{equation}
	\label{N2 union}
	\D^{I_N} = \bigcup_{r=1}^{\infty} U_{\epsilon,r}.
	\end{equation}
		
	We will first show that the sets in Definition \ref{N2 r sets} are closed for any $N$ separately symmetrically summable $(a_{(i_1, \dots, i_N)})_{(\cl{i}\in \prod_{n=1}^{N} I_n)}$.  Then we will show that if $(a_{(i_1, \dots, i_N)})_{(\cl{i}\in \prod_{n=1}^{N} I_n)}$ does not satisfy the conclusion of Theorem \ref{bf corollary}, there is $\epsilon>0$ such that for all $r>0$, $U_{r,\epsilon}$ is nowhere dense in $\mathbb{D}^{I_N}$.  As we have the union in line (\ref{N2 union}), this contradicts the Baire category theorem and we will be done.
	\end{remark}

	\begin{lemma}
	\label{N2 bf4}
	Suppose that $(a_{(i_1, \dots, i_N)})_{(\cl{i}\in \prod_{n=1}^{N} I_n)}$ is separately symmetrically summable. 
	Let $\lambda = (\lambda^{(1)}, \dots, \lambda^{(N-1)})$.
	Then for any $\epsilon,r>0$ the set $U_{\epsilon,r}$ of Definition \ref{N2 r sets} is closed.
	\end{lemma}
	\begin{proof}
Assume for contradiction that for some $\epsilon,r>0$, $U_{\epsilon,r}$ is not closed.  
Then there exists some $\eta\in \overline{U_{\epsilon,r}}\setminus U_{\epsilon,r}$.  
As $\eta\not\in U_{\epsilon,r}$, there exists a $\lambda \in \prod_{n=1}^{N-1} \D^{I_n}$ such that $a_{\lambda,\eta}$ cannot be $\epsilon$-$r$-approximated.  
Fix this $\lambda$.

Using (the contrapositive of) Lemma \ref{bf1}, part (i), there exists a finite rank projection $p\in \ell^\infty(X)$ such that $pa_{\lambda,\eta} p$ cannot be $\epsilon$-$r$-approximated.  

Now, for any $\mu\in \mathbb{D}^{I_N}$, the sum $\sum_{i\in I_N}\mu_ia_{\lambda,i}$ defining $a_{\lambda,\mu}$ is weakly convergent. 
As $p$ is finite rank, this implies that the sum $\sum_{i\in I_N}p \mu_i a_{\lambda,i} p$ is norm convergent.  Hence, using Lemma \ref{bf2}, for any $\delta>0$ there exists a finite subset $F$ of $I_N$ such that 
\begin{equation}\label{N2 db1}
\Bigg\|\sum_{i\in I_N\setminus F} p\mu_i a_{\lambda,i}p\Bigg\|<\delta
\end{equation}
for all $\mu\in \mathbb{D}^{I_N}$ (and in particular for $\mu=\eta$).

As $F$ is finite, the set
\begin{equation}\label{N2 db2}
\Bigg\{\mu\in \mathbb{D}^{I_N} ~\Big|~ |F|\max_{i\in F} \|a_{\lambda,i}\||\mu_i-\eta_i|<\delta \text{ for all } i\in F \Bigg\}
\end{equation}
is an open neighborhood of $\eta$ for the product topology. 
As $\eta$ is in the closure of $U_{\epsilon,r}$, the set in line \eqref{N2 db2} thus contains some $\theta\in U_{\epsilon,r}$. 
Hence in particular $pa_{\lambda,\theta} p$ is $\epsilon$-$r$-approximated, so there is $b\in \C^r_u[X]$ such that $\|pa_{\lambda,\theta} p -b\|\leq \epsilon$.

Note that 
$$
\|pa_{\lambda,\eta} p -b\|  \leq \|pa_{\lambda,\theta} p-b\|+\|pa_{\lambda,\eta} p -pa_{\lambda,\theta} p\| 
$$
$$
\leq \|pa_{\lambda,\theta} p-b\|+ \Bigg\|\sum_{i\in F} (\eta_i-\theta_i)pa_{\lambda,i} p\Bigg\|+\Bigg\|\sum_{i\in I_N\setminus F} \theta_i pa_{\lambda,i} p \Bigg\|+\Bigg\|\sum_{i\in I_N\setminus F} \eta_i pa_{\lambda,i} p \Bigg\|.
$$
The first term on the bottom line is bounded above by $\epsilon$ by choice of $b$, the second is bounded above by $\delta$ using that $\theta$ is in the set in line \eqref{N2 db2}, and the third and fourth terms are bounded above by $\delta$ using the estimate in line \eqref{N2 db1} (which is valid for all elements $\eta$ of $\mathbb{D}^{I_N}$).

Now, we have shown that for arbitrary $\delta>0$, we have found $b\in \C_u^r[X]$ such that $\|pa_{\lambda,\eta} p -b\|\leq \epsilon+3\delta$.  Using Lemma \ref{bf1}, part (ii), this implies that $pa_{\lambda,\eta} p$ can be $\epsilon$-$r$-approximated.  This contradicts our assumption in the first paragraph, so we are done.
\end{proof}

	\begin{lemma}
	\label{finite sum}
	Suppose that $(a_{(i_1, \dots, i_N)})_{(\cl{i}\in \prod_{n=1}^{N} I_n)}\subseteq \cu{X}$ is separately symmetrically summable.
	Let $\lambda = (\lambda^{(1)}, \dots, \lambda^{(N-1)})$.
	Then, for all $\epsilon>0$, for any $\theta \in \D^{I_N}$, and any finite $F \subseteq I_{N}$ there exists an $r>0$ such that the sum 
	$\sum_{i\in F}\theta_i a_{\lambda, i}$ is $\epsilon$-$r$-approximated.
	\end{lemma}
	
	\begin{proof}
	Let $F$ be a finite subset of $I_N$ and $\epsilon>0$ be given.
	By supposition, for each $i$, we may write  
	$$
	a_{\lambda, i} = b_{\lambda, i} + c_{\lambda, i} \text{ where } b_{\lambda, i} \in \cualg{r_i}{X} \text{ and } \norm{c_{\lambda, i}} < \tfrac{\epsilon}{\abs{F}}.
	$$
	Let $r = \max_{i\in F}\set{r_i}$ and note that $\sum_{i\in F} \theta_i b_{\lambda, i} \in \cualg{r}{X}$ for all $\lambda$.
	Additionally, 
	$$
	\norm{\sum_{i\in F} \theta_i c_{\lambda, i}} \leq \sum_{i\in F} \abs{\theta_i} \norm{c_{\lambda, i}} < \epsilon.
	$$
	Hence, $\sum_{i\in F} \theta_i a_{\lambda, i}$ is $\epsilon$-$r$-approximated for all $\lambda \in \prod_{n=1}^{N-1}\D^{I_n}$.
	\end{proof}
	
	\begin{lemma}
	\label{bf setminus}
	Suppose that $(a_{(i_1, \dots, i_N)})_{(\cl{i}\in \prod_{n=1}^{N} I_n)}$  is a separately symmetrically summable collection of operators in $\cu{X}$ that does not satisfy the conclusion of Lemma \ref{bf corollary}.
	Additionally, let $\lambda = (\lambda^{(1)}, \dots, \lambda^{(N-1)})$.
	Then there is an $\epsilon>0$ so that for all $r>0$ and all finite subsets $F \subseteq I_N$ there exists $\eta \in \mathbb{D}^{I_N}$ such that $\sum_{i\in I_N\setminus F} \eta_i a_{\lambda,i}$ cannot be $\epsilon$-$r$-approximated.
	\end{lemma}
	
	\begin{proof}
	Let $(a_{(i_1, \dots, i_N)})_{(\cl{i}\in \prod_{n=1}^{N} I_n)}$ be as in the statement.  
	Then there exists $\delta>0$ such that for all $r>0$ there exists $(\lambda, \eta)\in \opint{\prod_{n=1}^{N-1} \D^{I_n}} \times \mathbb{D}^{I_N}$ such that $a_{\lambda, \eta}$ is not $\delta$-$r$-approximable.  
	Fix this $\lambda$.
	Assume for contradiction that the conclusion of the lemma fails.  Then there exists $s>0$ and a finite subset $F$ of $I_N$ such that for all $\xi\in \mathbb{D}^{I_N}$ we have that $\sum_{i\in I_N\setminus F} \xi_i a_{\lambda,i}$ is $\delta/2$-$s$-approximated. 
	As $F$ is finite, by Lemma \ref{finite sum} there is a $t>0$ such that every element of
	$$
	\Bigg\{\sum_{i\in F}\xi_i a_{\lambda,i} ~\Big|~ \xi\in \mathbb{D}^{I_N}\Bigg\}
	$$
	can be $\delta/2$-$t$-approximated.  
	Now, for arbitrary $\xi\in \mathbb{D}^{I_N}$, 
	$$
	a_{\lambda,\xi} =\sum_{i\in F} \xi_i a_{\lambda,i} +\sum_{i\in I_N\setminus F} \xi_i a_{\lambda,i};
	$$
	as the first term above can be $\delta/2$-$s$-approximated, and as the second can be $\delta/2$-$t$-approximated, this implies that $a_{\lambda,\xi}$ can be $\delta$-$\max\{s,t\}$-approximated.  
	As $\xi$ was arbitrary, this contradicts the first sentence in the proof, and we are done.
\end{proof}
	
	As stated at the end of Remark \ref{rem ref}, the following lemma completes the proof of Corollary \ref{bf corollary}.

	\begin{lemma}
	Suppose that $(a_{(i_1, \dots, i_N)})_{(\cl{i}\in \prod_{n=1}^{N} I_n)}$  is a separately symmetrically summable collection of operators in $\cu{X}$ that does not satisfy the conclusion of Lemma \ref{bf corollary}.
	Let $\lambda = (\lambda^{(1)}, \dots, \lambda^{(N-1)})$.
	Then there is $\epsilon>0$ such that for each $r>0$ the set $U_{\epsilon,r}$ of Definition \ref{N2 r sets} is nowhere dense in $\mathbb{D}^{I_N}$.
	\end{lemma}

	\begin{proof}
	Let $(a_{(i_1, \dots, i_N)})_{(\cl{i}\in \prod_{n=1}^{N} I_n)}$ be as in the statement.  
	Then there exists $\delta>0$ such that for all $r>0$ there exists $(\lambda, \eta)\in \Big(\prod_{n=1}^{N-1} \D^{I_n}\Big) \times \mathbb{D}^{I_N}$ such that $a_{\lambda, \eta}$ is not $\delta$-$r$-approximable.  
	Fix this $\lambda$.
	Let $\epsilon'>0$ have the property from Lemma \ref{bf setminus}.  
	We claim that $\epsilon:=\epsilon'/2$ has the property required for this lemma.   
	Assume for contradiction that for some $r>0$, $U_{\epsilon,r}$ is not nowhere dense.  
	Lemma \ref{N2 bf4} implies that $U_{\epsilon,r}$ is closed, and so it contains a point $\xi$ in its interior.  
	Then by definition of the product topology there exists a finite set $F\subseteq I_N$ and $\delta>0$ such that the set 
	\begin{equation}\label{vdef 2}
	V:=\{\nu\in \mathbb{D}^{I_N}\mid |\xi_i-\nu_i|<\delta \text{ for all } i\in F\} \h \text{is contained in} \h U_{\epsilon,r}.
	\end{equation}

	Note that the element $\sum_{i\in F} \xi_i a_{\lambda,i}$ is in $C^*_u(X)$ by assumption, so can be $\epsilon$-$s$-approximated for some $s$. Let $b_{\xi, \lambda}\in \C_u^s[X]$ be such that $\|\sum_{i\in F} \xi_i a_{\lambda, i}-b_{\lambda,\xi}\|\leq \epsilon$.  
	On the other hand, Lemma \ref{bf setminus} gives us $\mu\in \mathbb{D}^{I_N}$ so that $\sum_{i\in I\setminus F} \mu_i a_{\lambda,i}$ cannot be $\epsilon'$-$\max\{r,s\}$-approximated.  
	We may further assume that $\mu_i=0$ for $i\in F$.  Define $\theta\in \mathbb{D}^I$ by 
	$$
	\theta_i:=\left\{\begin{array}{ll} \xi_i & i\in F \\ \mu_i & i\not\in F\end{array}\right.
	$$
	Then $\theta$ is clearly in the set $V$ of line \eqref{vdef 2}, and so $a_{\lambda,\theta}$ is $\epsilon$-$r$-approximated.  
	Let then $b_{\lambda,\theta}\in \C_u^r[X]$ be such that $\|a_{\lambda,\theta}-b_{\lambda,\theta}\|\leq \epsilon$.  
	We then see that  
	$$
	\|a_{\lambda,\mu}-(b_{\lambda,\theta}-b_{\lambda,\xi})\|  \leq  \|a_{\lambda,\mu}-a_{\lambda,\theta}+b_{\lambda,\xi}\|+\|a_{\lambda,\theta}-b_{\lambda,\theta}\|  
	$$
	$$\leq \Bigg\|b_{\lambda,\xi} - \sum_{i\in F} \xi_i a_{\lambda,i}\Bigg\|+\|a_{\lambda,\theta}-b_{\lambda,\theta}\| 
	$$
	The terms on the bottom row are each less than $\epsilon$ by choice of $b_{\lambda,\xi}$ and $b_{\lambda,\theta}$, and so $\|a_{\lambda,\mu}-(b_{\lambda,\theta}-b_{\lambda,\xi})\|\leq 2\epsilon=\epsilon'$.  
	As $b_{\lambda,\xi}+b_{\lambda,\theta}$ has propagation at most $\max\{r,s\}$, this contradicts the assumption that $a_{\lambda,\mu}$ cannot be $\epsilon'$-$\max\{r,s\}$-approximated, so we are done. 
	\end{proof}
	

\section{\centering Hochschild Cohomology}
\label{hochcoho}

In this section we introduce Hochschild cohomology, its construction, and several of its properties.

	\begin{definition}[Dual normal module]
	Let $\mathcal{M}$ be a von Neumann algebra.
	we say that $\mathcal{W}$ is a \textit{dual normal module} over $\mathcal{M}$ if:
	\begin{enumerate}[(i)]
	\item
	$\mathcal{W}$ is a dual $\mathcal{M}$-bimodule (Definition \ref{dual module}),
	\item
	and the maps 
	$$
	\mathcal{M} \to \mathcal{W} \h \text{defined by} \h m\mapsto mx \h \text{and} \h m\mapsto xm
	$$
	 are ultraweak - weak* continuous for all $x\in \mathcal{W}$.
	 \end{enumerate}
	 \end{definition}

	\begin{definition}[Subdual]
	Let $\mathcal{A}$ be a C*-algebra and let $\mathcal{V}$ be an $\mathcal{A}$-submodule of a dual module (as in Definition \ref{dual module}) $\mathcal{W}$ (under the same action). We will call such a module $\mathcal{V}$ a \textit{subdual} of $\mathcal{W}$. Note that we are not requiring  $\mathcal{V}$ to be a dual space, just that it is a submodule of a specified dual space. Moreover, if $\mathcal{A}$ is a C*-subalgebra of a von Neumann algebra $\mathcal{M}$ where $\mathcal{W}$ is a dual normal $\mathcal{M}$-module and the action of $\mathcal{A}$ on $\mathcal{V}$ is inherited from the $\mathcal{M}$-action on $\mathcal{W}$ then we say that $\mathcal{V}$ is a \textit{subdual normal} $\mathcal{A}$-module of $\mathcal{W}$.
	\end{definition}
	
	An example of a subdual normal module is the uniform Roe algebra acting on itself by multiplication. $\cu{X}$ acts on $\BB(\ell^2(X))$ by multiplication making $\BB(\ell^2(X))$ a $\cu{X}$-module. $\BB(\ell^2(X))$ is a dual space with predual $\mathcal{L}^1(\ell^2(X))$, the trace class operators. So $\cu{X}$ is a submodule of the dual space $\BB(\ell^2(X))$. However, $\cu{X}$ is not usually a dual space. This additional structure on the submodule allows us to use the relative weak* topology inherited from the parent module.
	
	By $\LL^n_c(\mathcal{A}, \mathcal{V})$ we mean the vector space of separately norm continuous multilinear maps from the $n$-fold Cartesian product of $\mathcal{A}$ to the $\mathcal{A}$-bimodule $\mathcal{V}$ when $n \geq 1$ and $\LL^0_c(\mathcal{A}, \mathcal{V}) := \mathcal{V}$. 
	
	Let $\mathcal{A}$ be a concrete C*-algebra.
	If $\mathcal{W}$ is a dual normal $\mathcal{A}$-bimodule with subdual $\mathcal{V}$, we use the notation 
	$\LL_w^n(\mathcal{A}, \mathcal{V})$ to indicate the vector space of 
	multilinear maps that are separately ultraweak-weak* continuous; that is, for $\phi \in \LL_w^n(\mathcal{A}, \mathcal{V})$
	$$
	\text{if} \h \set{a_{\alpha}} \subset\mathcal{A} \h \text{is a net such that}  \h a_{\alpha} \to a \in \mathcal{A} \h \text{ultraweakly in} \h\BB(\HH)   
	$$
	$$
	\text{then} \h \phi(\dots, a_{\alpha},\dots) \to \phi(\dots, a,\dots)\in \mathcal{V} \h \text{weak* in } \mathcal{W}.
	$$
		
	When we write $\LL^n(\mathcal{A},\mathcal{V})$ then either subscript may be attached.
	Considering $\A$ as a module over itself we will simply write $\LL(\A)$.
	Additionally, we equip both $\LL^n_c(\mathcal{A},\mathcal{V})$ and $\LL^n_w(\mathcal{A},\mathcal{V})$ with the operator norm.

	\begin{remark}
	Note that while $\LL^n_c(\mathcal{A},\mathcal{V})$ is complete in norm, we are not assuming, nor do we require these vector spaces to be complete in norm. \\
	\end{remark}
	
	To define the Hochschild cohomology we first construct the cochain complex
	$$
	0 \to \LL^{0}(\mathcal{A},\mathcal{V}) \os{\partial}{\to} \LL^{1}(\mathcal{A},\mathcal{V}) \os{\partial}{\to} \cdots
	\os{\partial}{\to} \LL^{n}(\mathcal{A},\mathcal{V}) \os{\partial}{\to} \LL^{n+1}(\mathcal{A},\mathcal{V}) \os{\partial}{\to} \cdots
	$$
	for both the norm continuous and ultraweak-weak* continuous cases
	where the coboundary operator $\partial: \LL^n(\mathcal{A}, \mathcal{V}) \to \LL^{n+1}(\mathcal{A}, \mathcal{V})$ is defined by

	$$
	(\partial\phi)(a_1, \dots, a_{n+1}) = a_1\phi(a_2, \dots, a_{n+1})
	$$
	$$
	+\sum_{j=1}^n (-1)^j\phi(a_1, \dots, a_ja_{j+1}, \dots, a_{n+1})
	$$
	$$
	+(-1)^{n+1} \phi(a_1, \dots, a_n)a_{n+1} \h (n \geq 1)
	$$
	and for $n=0$
	$$
	(\partial v)(a) = av-va \h (v\in\mathcal{V}, a \in \mathcal{A}).
	$$
	A straightforward calculation shows that $\partial^2$ is always zero.
	The $n^{th}$ Hochschild cohomology group $H_c^n(\mathcal{A}, \mathcal{V})$ (resp. $H_w^n(\mathcal{A}, \mathcal{V})$ in the ultraweak-weak* case) is the quotient vector space 
	$$
	H^n(\mathcal{A},\mathcal{V}):= \frac{\ker(\partial: \mathscr{L}^n(\mathcal{A}, \mathcal{V}) \to \mathscr{L}^{n+1}(\mathcal{A}, \mathcal{V}))}{\im(\partial: \mathscr{L}^{n-1}(\mathcal{A}, \mathcal{V}) \to \mathscr{L}^{n}(\mathcal{A}, \mathcal{V}))}.
	$$
		
	Additionally, when we consider $\A$ as a module over itself we simply write $H^n(\A)$.
	The cohomology obtained from this construction is the \textit{Hochschild} cohomology.
	We call an element $\phi \in \ker(\partial: \mathscr{L}^n(\mathcal{A}, \mathcal{V}) \to \mathscr{L}^{n+1}(\mathcal{A}, \mathcal{V}))$ a \textit{cocycle}, 
	and we call an element $\psi \in \im(\partial: \mathscr{L}^{n-1}(\mathcal{A}, \mathcal{V}) \to \mathscr{L}^{n}(\mathcal{A}, \mathcal{V}))$ a \textit{coboundary}.
	
	\begin{definition}[multimodular maps]
	Let $\mathcal{A}$ be a C*-algebra and let $\phi:\mathcal{A}^n \to \mathcal{V}$ be a bounded multilinear map to the Banach $\mathcal{A}$-bimodule $\mathcal{V}$. 
	If $\mathcal{B}$ is a C*-subalgebra of $\mathcal{A}$ we say that $\phi$ is \textit{$\mathcal{B}$-multimodular} if for any $b \in \mathcal{B}$ the following hold.
	\begin{enumerate}
	\item
	$
	b \phi(a_1, \dots, a_n) = \phi(ba_1, \dots, a_n),
	$
	\item
	$
	\phi(a_1, \dots, a_{j-1}b, a_j, \dots, a_n) = \phi(a_1, \dots, a_{j-1}, ba_j, \dots, a_n) \h \text{and}
	$
	\item
	$
	\phi(a_1, \dots, a_nb) = \phi(a_1, \dots, a_n)b
	$
	\end{enumerate}
	\end{definition}

	If $\mathcal{B}$ is a C*-subalgebra of $\mathcal{A}$ we use the notation $\LL^n(\mathcal{A}, \mathcal{V}: \mathcal{B})$ to indicate that the maps are $\mathcal{B}$-multimodular where we may use either subscript,``$c$" or ``$w$".
	As before we may construct the Hochschild cohomology of $\mathcal{B}$-multimodular maps which we denote by $H^n(\mathcal{A},\mathcal{V}:\mathcal{B})$ where either subscript $c$ or $w$ may be attached.
	Additionally, if we are considering $\A$ as a module over itself we simply write $H^n(\A:\mathcal{B})$.

	\subsection{Sinclair and Smith's `Reduction of Cocycles'}
	
	In this subsection we introduce a method to modify a cocycle, say $\phi \in \LL^n_c(\mathcal{A},\mathcal{W})$, by a coboundary to obtain a map in $\LL^n_c(\mathcal{A},\mathcal{W}:\mathcal{B})$ where $\mathcal{A} \sub\BB(\HH)$ is a C*-algebra, $\mathcal{B}$ is a C*-subalgebra of $\mathcal{A}$, and $\mathcal{W}$ is a dual normal $\BB(\HH)$-bimodule.

	\begin{lemma}[\cite{sinclair1995hochschild} Lemma 3.2.1]
	\label{lemma3.2.1}
	Let $\mathcal{B}$ be a unital subalgebra of a unital C*-algebra $\mathcal{A}$. Let $\mathcal{W}$ be a Banach $\mathcal{A}$-bimodule, and let $\phi \in \LL^n(\mathcal{A}, \mathcal{W})$ with $\partial\phi = 0$. 
	
	Then for all $b \in \mathcal{B}$ and $x_1,\dots, x_n \in \mathcal{A}$ we have:
	\begin{enumerate}[(i)]
	\item
	$\phi(b, x_2, \dots, x_n) =0$ if and only if\\
	$\phi(1,x_2, \dots, x_n) = 0$ and $\phi(bx_1, x_2, \dots, x_n) = b\phi(x_1,\dots, x_n)$.
	\item
	Fix $k \leq n$. Then for all $j\in\set{2,\dots, k}$,\\
	$\phi(x_1, \dots, x_{j-1}, b, x_{j+1}, \dots, x_n) =0$ if and only if\\
	$\phi(x_1, \dots, x_{j-1}, 1, x_{j+1}, \dots, x_n) = 0$ and\\ 
	$\phi(x_1, \dots, x_{j-1}b,x_{j}, \dots, x_n) = \phi(x_1, \dots, x_{j-1}, bx_{j}, \dots, x_n)$
	\item
	Additionally, \\
	$\phi(x_1, \dots, x_{n-1},b) =0$ if and only if\\
	$\phi(x_1, \dots, x_{n-1},1) =0$ and $\phi(x_1, \dots, x_nb) = \phi(x_1, \dots, x_n)b$ 
	\qed
	\end{enumerate}
	\end{lemma}

	\begin{lemma}[\cite{sinclair1995hochschild} Lemma 3.2.4]
	\label{lemma3.2.4}

	Let $\mathcal{B}$ be a C*-subalgebra spanned by an amenable group $\mathcal{U}$ (with respect to the discrete topology) of unitaries in a unital C*-algebra $\mathcal{A}$, and let $\mathcal{W}$ be a dual Banach $\mathcal{A}$-bimodule. There is a continuous linear map 
	$$
	K_n:\LL^n_c(\mathcal{A},\mathcal{W}) \to \LL^{n-1}_c(\mathcal{A},\mathcal{W}) 
	$$
	(depending on a choice of invariant mean on $ \mathcal{U}$)
	such that if $\phi \in \LL^n_c(\mathcal{A}, \mathcal{W})$ satisfies $\partial\phi = 0$ then $\phi-\partial(K_n\phi)$ is $\mathcal{B}$-multimodular.
	Moreover, we have that
	$$
	\norm{K_n}\leq\tfrac{(n+2)^n - 1}{n+1}.
	$$
	\qed
	\end{lemma}
	
	\begin{remark}
	As we will need it later, let us recall that the map $K_n$ is constructed recursively via 
	$$ 
	J_1: \LL_c^n(\mathcal{A},\WW) \to \LL_c^{n-1}(\mathcal{A},\WW) \h \text{defined by}  
	$$
	\begin{equation}
	\label{J_1}
	(J_1\phi)(a_1, \dots, a_{n-1}) = \int_{\UU}u^*\phi(u,a_1, \dots, a_{n-1})\dd\mu(u),
	\end{equation}
	$$
	G_k: \LL_c^n(\mathcal{A},\WW) \to \LL_c^{n-1}(\mathcal{A},\WW) \h \text{defined by}  
	$$
	\begin{equation}
	\label{G_k}
	(G_k\phi)(a_1, \dots, a_{n-1}) = \int_{\UU}\phi(a_1, \dots, a_ku^*,u, a_{k+1}, \dots, a_{n-1})\dd\mu(u),
	\end{equation}
	$$
	J_{k+1}:\LL_c^n(\mathcal{A},\WW) \to \LL_c^{n-1}(\mathcal{A},\WW) \h \text{defined by} \h 
	J_{k+1} = J_k+(-1)^kG_k(I-\partial J_k),
	$$
	\begin{equation}
	\label{K_n}
	\text{ and } K_n = J_n.
	\end{equation}
	\end{remark}
	
	\begin{lemma}[\cite{sinclair1995hochschild} Lemma 3.2.6]
	\label{lemma3.2.6}
	Let $\mathcal{A}$ be a unital C*-algebra and let $\mathcal{W}$ be a dual $\mathcal{A}$-bimodule.
	Suppose that $\mathcal{B}$ is a C*-subalgebra of $\mathcal{A}$ generated by an amenable group $\mathcal{U}$ of unitaries. 
	Then there is a continuous surjective linear projection $Q_n: \LL^n_c(\mathcal{A},\mathcal{W}) \to \LL^n_c(\mathcal{A},\mathcal{W}:\mathcal{B})$ such that $\partial Q_{n-1} = Q_n\partial$ and $\norm{Q_n}=1$.
	\qed
	\end{lemma}

	We conclude this section with a theorem that will be useful in the next section.
	
	\begin{theorem}[\cite{sinclair1995hochschild} Theorem 3.2.7]
	\label{theorem3.2.7}
	Let $\mathcal{B}$ be the C*-algebra generated by an amenable group $\mathcal{U}$ of unitaries in a unital C*-algebra $\mathcal{A}$, and let $\mathcal{W}$ be a dual $\mathcal{A}$-bimodule. Then
	$$
	H_c^n(\mathcal{A},\mathcal{W}) \iso H_c^n(\mathcal{A},\mathcal{W}: \mathcal{B})
	$$
	for all $n\in\N$ with isomorphism induced by the natural embedding 
	$$
	\LL_c^n(\mathcal{A},\mathcal{W}: \mathcal{B}) \hookrightarrow \LL_c^n(\mathcal{A},\mathcal{W}).
	$$
	\end{theorem}

	\begin{proof}
	Clearly, the natural embedding $\LL_c^n(\mathcal{A},\mathcal{W}: \mathcal{B}) \hookrightarrow \LL_c^n(\mathcal{A},\mathcal{W})$ induces a homomorphism $H_c^n(\mathcal{A},\mathcal{W}: \mathcal{B}) \rightarrow H_c^n(\mathcal{A},\mathcal{W})$. 
	By Lemma \ref{lemma3.2.4} this map is surjective. 
	Furthermore, if $\phi \in \LL_c^n(\mathcal{A},\mathcal{W}:\mathcal{B})$ and $\psi \in \LL_c^n(\A,\mathcal{W})$ is such that $\phi = \partial\psi$, then with $Q_n$ as in Lemma \ref{lemma3.2.6},
	$$
	\phi = Q_n\phi = Q_n\partial\psi = \partial Q_{n-1} \psi \h \text{where} \h Q_{n-1} \psi \in \LL_c^{n-1}(\mathcal{A},\mathcal{W}:\mathcal{B})
	$$
	and so our map is injective. 	\qedhere
	\end{proof}
	
	\begin{remark}
	If our averaging operator, i.e. the ``integral" over the unitary group $\mathcal{U}$, converges in the weak* topology of the dual normal $\mathcal{A}$-bimodule $\WW$ to an element in the subdual $\VV$ of $\WW$ for all $\phi \in \LL^n(\mathcal{A},\WW)$ then we may replace $\WW$ with $\VV$ everywhere above.
	\end{remark}


\section{\centering A Relation Between Cohomologies}
\label{sec5}
	
	The goal of this section is to prove the following theorem.
	\begin{theorem}
	\label{main 2}
	The natural map $H^n_w(\cu{X}) \to H^n_c(\cu{X})$ is surjective if and only if
	$$ H^n_c(\cu{X}) = 0.$$
	\end{theorem}
	
	Note that, by \cite{10.2307/1970433} Lemma 3, all bounded derivations on any C*-algebra are weakly continuous.
	Thus, the natural map $H^1_w(\mathcal{A}) \to H^1_c(\mathcal{A})$ is automatically surjective.
	However, this does not seem to be known for $n \geq 2$.	
	Additionally, our proof of Theorem \ref{main 2} depends on the underlying C*-algebra being a uniform Roe algebra.
	That being said, Theorem \ref{main 2} strictly generalizes Rufus Willett and the author's work in \cite{lorentz2020}.

	For notational convenience throughout we let: $A=\cu{X},~\BB=\BB(\ell^2(X))$, and $\ell = \ell^{\infty}(X)$.
	As a first step towards showing Theorem \ref{main 2} we show that $H_c^n(A,\BB:\ell) \iso H_c^n(A:\ell)$ in the following lemma.
	
	\begin{lemma}
	\label{multimodularequal}
	Let $\phi \in \LL^n(A,\BB:\ell)$. 
	Then $\phi$ takes image in the uniform Roe algebra; that is,
	$\LL^n(A,\BB:\ell) =\LL^n(A:\ell)$.
	\end{lemma}
	
	\begin{proof}
	Let $\phi \in \LL^n(A,\BB:\ell),~(x_1,\dots,x_n) \in A^n$, and $0<\epsilon\leq1$ be given.
	Set $M=\maxset{\norm{x_i}}+1$ and note that since each $x_i\in A$ we may write each $x_i$ as 
	$$
	x_i=a_i+b_i \h\text{where}\h a_i\in \cualg{r_i}{X} \h\text{and}\h \norm{b_i} 
	< \minset{\frac{\epsilon}{n\norm{\phi}M^n},~\epsilon}.
	$$
	Moreover,
	we have that $\norm{a_i}<M$.
	Next, since $\phi$ is multilinear we may write 
	$$
	\phi(x_1,\dots,x_n)
	=\phi(a_1,\dots,a_n) + \phi(a_1,\dots,a_{n-1},b_n) + \phi(a_1, \dots, a_{n-2}, b_{n-1},x_n) + 
	$$
	$$
	\dots +\phi(a_1,b_2,x_3,\dots,x_n) + \phi(b_1,x_2,\dots,x_n)
	$$
	Observe that every term but the first in this expansion has a $b_i$ in a single coordinate and either 
	$a_i$'s or $x_i$'s in the remaining coordinates.
	Thus,
	the norm for each of the terms with a $b_k$ in the $k$th coordinate is bounded by 
	$$
	\norm{\phi}\opint{\prod_{i=1}^nM}\norm{b_k} < \frac{\epsilon}{n}
	$$
	Hence,
	it is enough to show that $\phi(a_1,\dots,a_n) \in \cualg{n\cdot r}{X}$ where $r=\maxset{r_i}$.

	To show this
	let $p_x$ be the projection onto the span of the Dirac mass at $x$, and
	let $B_x(r)$ denote the closed ball of radius $r$ centered at $x$.
	We then define
	$$
	p_{B_x(r)} := \sum_{k \in B_x(r)}p_k.
	$$
	Note that, the sum defining $p_{B_x(r)}$ is finite for any given $r \in \N$ since $X$ has bounded geometry.
	Next, for any fixed $x\in X$, 
	\begin{equation}
	\label{sum projection}
	p_x a_1 = p_x a_1 p_{B_x(r)} \h \text{and} \h p_{B_x((i-1)\cdot r)} a_i = p_{B_x((i-1)\cdot r)} a_i p_{B_x(i\cdot r)}
	\end{equation}
	since each $a_i$ has propagation less than $r$. 
	Next, fix $x,y \in X$ such that $d(x,y) >n \cdot r$ and observe that
	$$
	p_x \phi(a_1, \dots, a_n) p_y = \phi(p_xa_1, \dots, a_np_y)
	$$
	$$
	=\phi(p_xa_1 p_{B_x(r)} , \dots, a_np_y) = \phi(p_xa_1 p_{B_x(r)} ,p_{B_x(r)}a_2, \dots, a_np_y)
	$$
	where on the left hand we have used line (\ref{sum projection}) and on the right hand side we use that $\phi$ is $\ell^{\infty}(X)$-multimodular.
	
	Continuing this process $n-1$ times we arrive at
	$$
	p_x\phi(a_1, \dots, a_n)p_y
	$$
	$$
	=\phi(p_xa_1 p_{B_x(r)} ,\dots,  p_{B_x((i-1) \cdot r)} a_i  p_{B_x(i\cdot r)}   \dots,p_{B_x((n-1)\cdot r)} a_np_y).
	$$
	
	Observe that for any $k\in B_x((n-1)\cdot r)$,
	$$
	d(k,y) \geq d(x,y)-d(x,k) \geq d(x,y) -  (n-1)\cdot r > n\cdot r - (n-1)\cdot r = r,
	$$
	and so 
	$$
	p_{B_x((n-1)\cdot r)} a_np_y = 0 \h \text{since} \h a_n \in \cualg{r}{X}.
	$$
	Thus, 
	$$
	p_x\phi(a_1, \dots, a_n)p_y = 0
	$$
	and since $x,y \in X$ were an arbitrary pair satisfying $d(x,y) > n\cdot r$, we have that $\phi(a_1, \dots, a_n) \in \cualg{n\cdot r}{X}$ as was to be shown.
	\end{proof}

	\begin{remark}
	\label{key}
By Lemma \ref{multimodularequal} and Theorem \ref{theorem3.2.7} we know that 
$$
H_c^n(A:\ell)\iso H_c^n(A,\BB:\ell) \iso H_c^n(A,\BB).
$$
In Sinclair and Smith \cite{sinclair1995hochschild} Theorem 3.3.1 they show that $H_c^n(A,\BB) \iso H_c^n(\BB)$, which we also show in the sequel, Remark \ref{all of these too}.
Hence, by Theorem \ref{iszero} $H_c^n(A:\ell) = 0$. 
Thus, we need only show that the homomorphism 
$$
H_c^n(A:\ell) \to H_c^n(A) \text{ induced by the inclusion } \LL_c^n(A:\ell) \to \LL_c^n(A)
$$ 
is a surjection. 
By Lemma \ref{lemma3.2.4}, averaging over the unitary group of $\ell^{\infty}(X)$, we know that for a cocycle $\phi\in\LL_c^n(A)$, $(\phi-\partial K_n\phi) \in \LL_c^n(A:\ell)$.
Thus, to show that $H_c^n(A:\ell) \to H_c^n(A)$ is a surjection it suffices to show that $K_n\phi \in \LL_c^{n-1}(A)$ so that  $\partial K_n\phi$ is a coboundary in $\LL_c^n(A)$, for then 
$$
H_c^n(A:\ell) \ni [\phi-\partial K_n\phi] = [\phi] \text{ in } H_c^n(A).
$$
Furthermore, since $H_w^n(A) \to H_c^n(A)$ is a surjection by the hypothesis of Theorem \ref{main 2}, we may assume that $\phi \in \LL_w^n(A)$.
	\end{remark}

	Before we embark on the proof that $H^n_c(A) = 0$ if the map $H_w^n(A) \to H_c^n(A)$ is a surjection for a general $n$, we show some properties of the map $K_n$ arising from its construction and set some notation.

	\begin{lemma}
	$K_n$ is the sum of $\sum_{k=1}^n 2^{k-1}$ terms (before applying the boundary operator), where the first term is $J_1$, the next terms are the  $n$-alternating sum of the maps $G_k$, and the remaining terms for $n\geq2$ are of the form
	\begin{equation}
	\label{likethis}
	G_{j_i}\partial \dots G_{j_1}\partial J_1 \h \text{or} \h G_{j_i}\partial \dots G_{j_2}\partial G_{j_1} \h \text{for} \h
	j_i > j_{i-1} > \dots > j_1.
	\end{equation}
	
	\end{lemma}
	
	\begin{proof}
	Since $K_n$ is defined by $K_n = J_n$ where $J_{k+1} = J_k + (-1)^k (G_k - G_k\partial J_k)$ we will induct on $k$.
	
	Let $D_k = (G_k - G_k\partial J_k)$, then
	$$
	J_{k+1} = J_k +(-1)^k D_k
	$$
	$$
	= J_{k-1} + (-1)^{k-1}D_{k-1} + (-1)^k D_k
	$$
	$$
	= J_1 + \sum_{j=1}^k (-1)^j D_j
	$$
	$$
	= J_1 + \sum_{i=1}^k (-1)^i G_i + \sum_{j=1}^k (-1)^{j+1}G_j \partial J_j
	$$
	Note that, since $j \leq k$ for all $j$ in the last summation, by inductive hypothesis our terms are of the form of line (\ref{likethis}).

	Lastly, using the recursive definition of $J_{k+1}$ and letting $\abs{J_{k+1}}$ be the number of terms of $J_{k+1}$,we have
	$$
	\abs{J_{k+1}} = \abs{J_k} + \abs{G_k} + \abs{G_k\partial J_k} = 2\abs{J_k} + 1 = 2\sum_{j=1}^k 2^{j-1} +1 = \sum_{j=1}^{k+1} 2^{j-1}
	$$
	as was to be shown.
	\end{proof}

	\begin{lemma}
	\label{finite K_n sum}
	Let $\phi\in \LL^n(A)$ and let $(a_1, \dots, a_n), ~ a_i \in A$ be given. Then
	$$
	(G_{j_i}\partial \dots G_{j_1}\partial J_1\phi)(a_1,\dots, a_{n-1}) \h \text{and} \h 
	(G_{j_i}\partial \dots G_{j_2}\partial G_{j_1}\phi)(a_1, \dots, a_{n-1}) 
	$$
	are both finite sums of terms of the form
	$$
	\scalebox{0.85}{$\ds\int_{\UU}\dots\int_{\UU}\prod_{k=1}^N(c_{1,k}v_{1,k})\phi\Big(\prod_{k=1}^Nc_{2,k}v_{2,k}, \dots, \prod_{k=1}^Nc_{n,k}v_{n,k}\Big)\prod_{k=1}^Nc_{n+1,k}v_{n+1,k}~ \dd\mu(u_{j_i})\dots \dd\mu(u_{j_1})$}
	$$
	where each $c_{\ell,k} $ is fixed as one of the $a_j$'s or 1, and $v_{\ell,k}\in\UU$ the unitary group of $\ell^{\infty}(X)$.
	Additionally $N <\infty$.
	\end{lemma}
	
	\begin{proof}
	Consider
	$$
	(G_{j_i}\partial \dots G_{j_1}\partial J_1\phi)(a_1,\dots, a_{n-1}).
	$$
	Observe that, after applying $G_{j_i}$, in the $l$'th coordinate we will have: $a_l,~ a_lu_{j_i}^*,$ or $u_{j_i}$.
	Note that we may write this coordinate as $c_lv_l$ where $c_l$ is fixed as 1 or $a_l$ and $v_l = 1,~u_{j_i},$ or $u_{j_i}^*$.
	Also note that $v_l\in\UU$.
	Thus we may write,
	$$
	(G_{j_i}\partial \dots G_{j_1}\partial J_1\phi)(a_1,\dots, a_{n-1})
	$$
	\begin{equation}
	\label{eq4}
	= \int_{\UU}(\partial G_{j_{i-1}} \dots \partial G_{j_1}\partial J_1\phi)(c_1v_1,\dots, c_nv_n)~\dd\mu(u_{j_i}).
	\end{equation}
	Next, since our averaging operator is finitely additive and the boundary operator introduces a finite number of terms, we may `bring in' the averaging operator to each term. 
	Additionally, since the boundary operator just moves one of the arguments to the coordinate to the left, in front of, or behind the map, we may write (after reindexing) a typical term obtained from applying the boundary map in line (\ref{eq4}) as
	$$
	\scalebox{0.98}{$\ds\int_{\UU} c_0v_0( G_{j_{i-1}}\partial \dots  G_{j_1}\partial J_1\phi)(c_1v_1c_2v_2,\dots, c_{2n}v_{2n}c_{2n+1}v_{2n+1})c_{2n+2}v_{2n+2}~\dd\mu(u_{j_i})$}
	$$
	where $c_k \in \set{1,a_1,\dots,a_{n-1}}$ and $v_k\in \UU$ (note that the $c_k$'s will be fixed differently for each term).
	Applying this process again it is not hard to see that after applying $G_{j_i}\partial G_{i-1} \partial$ we will have a finite sum of terms of the form
	$$
	\scalebox{0.80}{$\ds\int_{\UU}\int_{\UU}\prod_{k=1}^4(c_{0,k}v_{0,k})(G_{j_{i-2}}\dots\partial J_1\phi)\Big(\prod_{k=1}^4c_{1,k}v_{1,k}, \dots, \prod_{k=1}^4c_{n,k}v_{n,k}\Big)\prod_{k=1}^4c_{n+1,k}v_{n+1,k}~ \dd\mu(u_{j_i}) \dd\mu(u_{j_{i-1}})$}.
	$$
	Note that the application of the $J_1$ map does not change our technique and eventually this process must end.
	Thus, the conclusion holds and we are done.
	\end{proof}

	\begin{definition}
	\label{fixing}
	For each term obtained in the previous lemma the set of $\set{c_{l,k}}$ is fixed for that term. 
	We shall call this a \textit{partial coordinate fixing} of $\phi$.
	\end{definition}
	
	\begin{lemma}
	\label{general phi}
	Let $(a_1, \dots, a_n), ~ a_i \in A$ and $\phi \in \LL_w^{n}(A)$ be given.
	Consider
		\begin{equation}
		\label{phi general}
		y_0\phi(y_1, \dots , y_n) y_{n+1}, \h y_i = \prod_{j=1}^{N_i}c_j f_{(i,j)} , \h \text{where } N_i < \infty
		\end{equation}
	where $f_{(i,j)}$ is any element in $\opint{\ell^{\infty}(X)}_1$, and each $c_j = a_k$ or $1$ is fixed. 
	Then for all $\epsilon>0$ there exists an $r>0$ (depending on the partial coordinate fixing of $\phi$) such that $y_0\phi(y_1, \dots , y_n) y_{n+1}$ can be $\epsilon$-$r$-approximated.
	
	\end{lemma}

	\begin{proof}
	Let $p_x \in \BB(\ell^2(X))$ be the rank one projection onto the span of the Dirac mass at $x$. 
	For any element $f$ in the unit ball of $\ell^{\infty}(X)$, we may write $f$ as a strongly (and so weakly) convergent sum
		\begin{equation}
		\label{sot sum}
		f= \sum_{x\in X} f(x)p_x.
		\end{equation}
	
	Then, for an arbitrary $i,j$ where $1\leq i \leq N$ and $0 \leq j \leq n+1$ and $f_{(\ell,k)} \in (\ell^{\infty}(X))_1$ fixed whenever $\ell,k \neq i,j$, we have that
	$$
	\sum_{x_j \in X} \lambda^{(j)}_{x_j}y_0\phi\opint{y_1, \dots, \opint{\prod_{k=1}^{j-1}c_k f_{(i,k)}}c_jp_{x_j}\opint{\prod_{k=j+1}^{N_i}c_k f_{(i,k)}}, \dots , y_n}y_{n+1}
	$$
	weakly converges to
	$$
	y_0\phi(y_1, \dots, \prod_{k=1}^{N_i}c_k f_{(i,k)},\dots , y_n)y_{n+1}
	$$
	Moreover, (\ref{phi general}) is bounded above by $\norm{\phi}\prod_{k=1}^n\norm{a_k}$ for all $f_{(\ell,k)} \in \opint{\ell^{\infty}(X)}_1$.
	Hence, since the weak and ultraweak topologies coincide on norm bounded sets and $\phi\in\LL_w^n(A)$, we have that, for each partial coordinate fixing of $\phi$,
	$$
	\prod_{k=1}^{N_0}(c_{0,k}p_{x_{(0,k)}})\phi\Big(\prod_{k=1}^{N_1}c_{1,k}p_{x_{(1,k)}}, \dots, \prod_{k=1}^{N_n}c_{n,k}p_{x_{(n,k)}}\Big)\prod_{k=1}^{N_{n+1}}c_{n+1,k}p_{x_{(n+1,k)}}
	$$
	is separately symmetrically summable.
	Thus, by Corollary \ref{bf corollary}, for all $\epsilon>0$ there exists an $r>0$ (depending on the partial coordinate fixing of $\phi$) such that $y_0\phi(y_1, \dots , y_n) y_{n+1}$ can be $\epsilon$-$r$-approximated.
	\end{proof}
	
	\begin{lemma}
	$K_n\phi \in \LL_c^{n-1}(A)$ whenever $\phi \in \LL_w^n(A)$, where $K_n$ is constructed by averaging over the unitary group of $\ell^{\infty}(X)$.
	\end{lemma}

	\begin{proof}
	Let $\epsilon>0$ be given.
	
	By Lemmas \ref{finite sum} and \ref{general phi}, $K_n\phi(a_1, \dots, a_{n-1})$ is the finite sum of finite sums of terms of the form 
	$$
	\int_{\mathcal{U}_{j_i}} \dots \int_{\mathcal{U}_{j_1}} y_0\phi(y_1, \dots , y_n) y_{n+1} \dd\mu(u_{j_1})\dots \dd\mu(u_{j_i})
	$$
	where each term is a different partial coordinate fixing of $\phi$ (in the sense of Definition \ref{fixing}).
	Using Lemma \ref{general phi} we may write each of these terms as
	\begin{equation}
	\label{breakdown}
	=\int_{\mathcal{U}_{j_i}} \dots \int_{\mathcal{U}_{j_1}} a(\cl{u}) + b(\cl{u}) \dd\mu(u_{j_1})\dots \dd\mu(u_{j_i})
	\end{equation}
	where each $a(\cl{u}) \in \cualg{r}{X}$ and $\norm{b(\cl{u})} < \epsilon/M$ for a given $M>0$ and all $\cl{u} \in \mathcal{U}_{j_i} \times \cdots \times \mathcal{U}_{j_1}$.
	Thus, taking $M$ and $R$ sufficiently large, since $K_n\phi(a_1, \dots, a_{n-1})$ is the finite sum of terms as in line (\ref{breakdown}), $K_n\phi(a_1, \dots, a_{n-1})$ is $\epsilon$-$R$-approximated. Since $\epsilon$ was arbitrary, we are done.
	\end{proof}
	
	\begin{proof}[Proof of Theorem \ref{main 2}]
	By Remark \ref{key}, to show that $H_c^n(\cu{X}) = 0$ it suffices to show that $K_n\phi \in \LL_c^{n-1}(\cu{X})$ whenever $\phi \in \LL_w^n(\cu{X})$, which we have done in the previous lemma.
	\end{proof}


\section{\centering Ultraweak-Weak* Continuous Cohomology}	
\label{sec6}
	In this section we discuss methods for relating norm continuous and ultraweak-weak* continuous cohomologies which will allow us to obtain the following result.
	
	 \begin{theorem}
	 \label{main 3}
	 If $H_c^n(\cu{X}) = 0$ for all $n \in \N$ then $ H_w^n(\cu{X}) = 0$ for all $n\in \N$.
	 \end{theorem}

	To accomplish the goals of this section we will have to use the \textit{enveloping von Neumann algebra}.
	
	\begin{theorem}[\cite{takesaki2013theory} III.2.2, III.2.4]
	\label{tilde pi}
	Let $\mathcal{A}$ be a C*-algebra and $\opint{\pi,\HH_{\pi}}$ be the universal representation of $\mathcal{A}$.
	Then there is a unique linear map $\tilde{\pi}$ of the double dual $\mathcal{A}^{**}$ onto  $\cl{\pi(\mathcal{A})}$, the weak closure of $\pi(A)$, with the following properties:

	\begin{enumerate}[(i)]
	\item
	If $\iota$ is the natural embedding then the diagram\\
	\begin{tikzpicture}
  \matrix (m) [matrix of math nodes,row sep=3em,column sep=8em,minimum width=2em]
  {
     \mathcal{A} & \cl{\pi(\mathcal{A})} \\
     \mathcal{A}^{**} & \\};
  \path[-stealth]
    (m-1-1) edge node [left] {$\iota$} (m-2-1)
    (m-1-1) edge node [above] {$\pi$} (m-1-2)
    (m-2-1) edge node [below] {$\tilde{\pi}$}  (m-1-2);
	\end{tikzpicture}\\
	is commutative.
	\item
	$\tilde{\pi}$ is $\sigma(\mathcal{A}^{**},\mathcal{A}^*)$-ultraweak continuous.
	\item
	$\tilde{\pi}$ maps the unit ball $(\mathcal{A}^{**})_1$ onto the unit ball $(\cl{\pi(\mathcal{A})})_1$.
	\item
	$\tilde{\pi}$ is a $\sigma(\mathcal{A}^{**},\mathcal{A}^*)$-ultraweak homeomorphism.
	\qed
	\end{enumerate}
	
	\end{theorem}

\subsection{A Bridge Between Ultraweak-Weak* Continuous and Norm Continuous Cohomology}

In this subsection we discuss a method to extend separately continuous multilinear maps to separately ultraweakly continuous multilinear maps. Many of the proofs can be found in Sinclair and Smith \cite{sinclair1995hochschild}.

\begin{lemma}[\cite{sinclair1995hochschild} 3.3.2]
\label{lem3.3.2}
Let $\mathcal{A}$ and $\mathcal{B}$ be C*-algebras acting nondegenerately on a Hilbert space $\HH$ with ultraweak closures $\cl{\mathcal{A}}$ and $\cl{\mathcal{B}}$. 
Let $\tau$ be a bounded bilinear form on $\mathcal{A} \times \mathcal{B}$.
If $\tau$ is separately ultraweakly continuous, then $\tau$ extends uniquely to to a separately ultraweakly continuous bilinear form $\cl{\tau}$ on $\cl{\mathcal{A}} \times \mathcal{B}$.
\qed
\end{lemma}
	
\begin{lemma}[\cite{sinclair1995hochschild} Lemma 3.3.3]
\label{weak extension}
Let $\mathcal{A}$ be a C*-algebra acting nondegenerately on a Hilbert space $\HH$ and let $\mathcal{V}$ be the dual of a Banach space $\mathcal{V}_*$.
If $\phi$ is a bounded $n$-linear map from $\mathcal{A}^n$ to $\mathcal{V}$ that is separately ultraweak-weak* continuous, then $\phi$ extends uniquely without change in norm to a bounded $n$-linear map $\cl{\phi}$ from $(\cl{\mathcal{A}})^n$ to $\mathcal{V}$ that is seperately ultraweak-weak* continuous.
\qed
\end{lemma}

	The following lemma can be found in Blackadar \cite{blackadar2006operator} III.5.2.11 or Takesaki \cite{takesaki2013theory} III.2.4, III.2.14. 
	\begin{lemma}
	\label{takesaki}
	Let $\mathcal{A}$ be a C*-algebra acting on a Hilbert space $\HH$ with weak closure $\overline{\mathcal{A}}$.  If $\pi$ is the universal representation of $\mathcal{A}$, then there is a projection $p$ in the center of the weak closure $\overline{\pi(\mathcal{A})}$ of $\pi(\mathcal{A})$ and a $*$-isomorphism
	$$
	\theta:p\overline{\pi(\mathcal{A})} \to \overline{\mathcal{A}} \h \text{such that}
	$$
	\begin{equation}
	\label{theta}
	\theta(p\pi(a))=a \h\text{and}\h \theta(px) = \pi^{-1}(x) \text{ for all } a\in \mathcal{A} \text{ and } x\in \pi(\mathcal{A}).
	\end{equation}
	Moreover, $\theta$ is a homeomorphism from $p\overline{\pi(\mathcal{A})}$ onto $\overline{\mathcal{A}}$ if both have their ultraweak topologies, since $*$-isomorphisms between von Neumann algebras are ultraweak homeomorphisms.
	\qed
	\end{lemma}
	
	\begin{lemma}[\cite{sinclair1995hochschild} Lemma 3.3.4]
	\label{lemma3.3.4}
	Let $\mathcal{A}$ be a C*-algebra acting on a Hilbert space $\HH$ with weak closure $\overline{\mathcal{A}}$.  Let $\pi$ be the universal representation of $\mathcal{A}$, and let $p$, $\theta$ be as in Lemma \ref{takesaki}. Additionally, let $\mathcal{V}$ be a dual normal $\overline{\mathcal{A}}$-module. Then $\mathcal{V}$ may be regarded as a dual normal $\overline{\pi(\mathcal{A})}$-module via 
	\begin{equation}
	\label{action}
	x\cdot v = \theta(px)v \h \text{and} \h v\cdot x = v\theta(px)
	\end{equation}
	and there are continuous linear maps 
	$$
	T_n: \LL_c^n(\mathcal{A}, \mathcal{V}) \to \LL_w^n(\overline{\pi(\mathcal{A})}, \mathcal{V}) ,
	$$
	$$
	S_n:  \LL_w^n(\overline{\pi(\mathcal{A})}, \mathcal{V}) \to  \LL_w^n(\overline{\mathcal{A}}, \mathcal{V})
	$$
	$$
	W_n:  \LL_w^n(\overline{\pi(\mathcal{A})}, \mathcal{V}) \to \LL_c^n(\mathcal{A}, \mathcal{V})
	$$
	such that:
	
	\begin{enumerate}[(i)]
	\item
	$\partial T_n = T_{n+1}\partial$, $\partial W_n = W_{n+1}\partial$, and $\partial S_n = S_{n+1}\partial$,
	\item
	$\norm{T_n}, \norm{S_n}, \norm{W_n} \leq 1$,
	\item
	if $\mathcal{B}$ is a C*-subalgebra of $\mathcal{A}$, $T_n$ maps $\mathcal{B}$-multimodular maps to $\overline{\pi(\mathcal{B})}$-multimodular maps, and $S_n$ and $W_n$ map $\overline{\pi(\mathcal{B})}$-multimodular maps to $\mathcal{B}$-multimodular maps,
	\item
	$S_nT_n$ is a projection from $\LL_c^n(\mathcal{A}, \mathcal{V})$ onto $\LL_w^n(\mathcal{A}, \mathcal{V})$,
	\item
	if $\mathcal{C}$ is the C*-subalgebra of $\overline{\pi(\mathcal{A})}$ generated by $1$ and $p$, and if 
	$$
	\psi \in \LL_w^n(\overline{\pi(\mathcal{A})}, \mathcal{V}: \mathcal{C}),
	$$
	then $W_n\psi = S_n\psi \in \LL_w^n(\overline{\mathcal{A}}, \mathcal{V}),$
	\item
	$W_nT_n$ is the identity map on $\LL_c^n(\mathcal{A}, \mathcal{V})$.
	\qed
	\end{enumerate}
	\end{lemma}
	
	Note that for $\psi \in \LL^n_w(\cl{\pi(\mathcal{A})}, \mathcal{V}:\mathcal{C})$ being $\mathcal{C}$-multimodular is equivalent to having the property that $\psi(a_1, \dots, a_n) = 0$ if any of the arguments $a_j \in (1 -p)\overline{\pi(\mathcal{A})}$. As we shall need it later, we review the construction of the maps $T_n,~S_n,$ and $W_n$.
	
	\begin{remark}[The map $T_n$]
	\label{T_n}
	The equation 
	\begin{equation}
	\phi_1(x_1, \dots, x_n)= \phi(\theta(px_1), \dots, \theta(px_n)) \h \text{for all} \h x_1, \dots, x_n \in \pi(\mathcal{A})
	\end{equation}
	defines $\phi_1 \in \LL_c^n(\pi(\mathcal{A}), \mathcal{V})$. 
	This map is separately ultraweakly-weak* continuous in each of its arguments, because $\pi$ is the universal representation of $\mathcal{A}$, so by \cite{takesaki2013theory} III.2.4 each continuous linear functional on $\pi(\mathcal{A})$ is ultraweakly continuous; that is $\phi_1 \in \LL_w^n(\pi(\mathcal{A}), \mathcal{V})$.
	Moreover, by \cite{sinclair1995hochschild} Lemma 3.3.3, we may extend $\phi_1$ to $\overline{\phi_1} \in \LL_w^n(\overline{\pi(\mathcal{A})}, \mathcal{V})$ without change of norm. The map $T_n$ is then defined by $T_n\phi = \overline{\phi_1}$. 
	\end{remark}
	
	\begin{remark}
	The map $S_n:  \LL_w^n(\overline{\pi(\mathcal{A})}, \mathcal{V}) \to  \LL_w^n(\overline{\mathcal{A}}, \mathcal{V})$ is defined by
	$$
	(S_n\psi)(a_1, \dots, a_n) = \psi(\theta^{-1}(a_1), \dots, \theta^{-1}(a_n)).
	$$
	\end{remark}

	\begin{remark}
	\label{W_n wins}
	The map $W_n: \LL_w^n(\overline{\pi(\mathcal{A})},\mathcal{V}) \to \LL_c^n(\mathcal{A},\mathcal{V})$ is defined by
	$$
	W_n\psi(a_1, \dots, a_n) = \psi(\pi(a_1), \dots, \pi(a_n)).
	$$
	Note that, if $\mathcal{A} \subseteq \mathcal{V} = \BB(\HH)$ then $W_n$ maps $\LL_w^n(\pi(\mathcal{A}),\mathcal{A})$ to $\LL_c^n(\mathcal{A})$.
	\end{remark}
	
	\begin{lemma}[\cite{sinclair1995hochschild} Lemma 3.3.5]
	Let $\A$ be a C*-algebra and let $\mathcal{V}$ be a dual normal $\A$-bimodule.
	Then the homomorphism 
	$$
	H_w^n(\A,\mathcal{V}) \to H_c^n(\A, \mathcal{V}) \h \text{induced by} \h \LL_w^n(\A,\mathcal{V}) \to \LL_c^n(\A,\mathcal{V})
	$$
	is surjective.
	\qed
	\end{lemma}
	
	\begin{theorem}[\cite{sinclair1995hochschild} Theorem 3.3.1]
	\label{theorem3.3.1}
	Let $\mathcal{A}$ be a C*-algebra acting on a Hilbert space $\HH$ with weak closure $\overline{\mathcal{A}}$. 
	Additionally, let $\mathcal{V}$ be a dual normal $\overline{\mathcal{A}}$-module.
	Then, 
	$$
	H_c^n(\A,\mathcal{V}) \iso H_w^n(\A,\mathcal{V}) \iso H_w^n(\cl{\A}, \mathcal{V})
	$$
	\end{theorem}
	
	\begin{proof}
	By the previous lemma we have that the map $H_w^n(\A,\mathcal{V}) \to H_c^n(\A,\mathcal{V})$ is a surjection.
	To see that this map is injective first note that for $\psi \in \LL_c^{n-1}(\A,\mathcal{V})$ we have that $S_{n-1}T_{n-1}\psi \in \LL_w^{n-1}(\A,\mathcal{V})$ by Lemma \ref{lemma3.3.4} (iv). 
	Next, if $\phi \in \LL_w^n(\A,\mathcal{V})$ with $\phi = \partial \psi$ where $\psi \in \LL_c^{n-1}(\A,\mathcal{V})$, then
	$$
	\phi = \partial \psi = S_nT_n \partial\psi = \partial S_{n-1}T_{n-1}\psi.
	$$
	Thus, the map $H_w^n(\A,\mathcal{V}) \to H_c^n(\A,\mathcal{V})$ induced by the inclusion is also injective and so an isomorphism. 
	
	Lastly, by Lemma \ref{weak extension}, the restriction map $\LL_w^n(\cl{\A}, \mathcal{V}) \to \LL_w^n(\A,\mathcal{V})$ is an isomorphism and so we are done.
	\end{proof}
	
	\begin{remark}
	\label{all of these too}
	Note that the ultraweak closure of $\cu{X}$ is $\BB(\ell^2(X))$.
	Hence, by the previous theorem we have that
	$$
	H_c^n(\cu{X},\BB(\ell^2(X))) \iso H_c^n(\BB(\ell^2(X))).
	$$
	\end{remark}

\subsection{On the Vanishing of the Ultraweak-Weak* Continuous Cohomology of Uniform Roe Algebras}

	 Before we prove Theorem \ref{main 3} we will need  a few lemmas. Once more for notational convenience throughout we let: $A=\cu{X},~\BB=\BB(\ell^2(X))$, and $\ell = \ell^{\infty}(X)$.
	 
	\begin{lemma}
	Let $\pi$ be the universal representation of $A$, and let $p$ be the projection from Lemma \ref{takesaki}. 
	If $\set{q_{\alpha}}$ is the net of finite rank projections in $\ell$ with its usual ordering then
	$$
	\pi(q_{\alpha}) \xrightarrow{\text{\tiny{ultraweakly}}} p \h \text{in} \h \cl{\pi(\ell)}.
	$$
	\end{lemma}
	
	\begin{proof}
	Recall that the double dual of the compact operators $\KK(\HH)^{**}$ is naturally identified with $\BB $ (cf. \cite{takesaki2013theory} II.1.8).
	Moreover, since $\KK(\HH)$ is an ideal in $\cu{X}$ and $\set{q_{\alpha}}$ is an approximate unit for $\KK(\HH)$, by Blackadar \cite{blackadar2006operator} III.5.2.11, there exists a central projection $q \in A^{**}$ such that 
	$$
	\hat{q}_{\alpha} \to q \h \text{in the} \h \sigma(A^{**},A^*) \h \text{topology and} \h qA^{**} = \KK(\HH)^{**} \iso p\cl{\pi(A)}
	$$
	
	Thus, if $\tilde{\pi}$ is the map from Lemma \ref{tilde pi}, using Lemma \ref{tilde pi}, we have that $\tilde{\pi}(q) = p$.
	Moreover, since $\set{q_{\alpha}} \subseteq \ell$ and $\tilde{\pi}$ is a $\sigma(A^{**},A^*)$-ultraweak homeomorphism,
	we have that 
	$$
	\pi(q_{\alpha}) \xrightarrow{\text{\tiny{ultraweakly}}} p \h \text{and} \h p \in \cl{\pi(\ell)}.         
	$$
	\end{proof}
	
	
	\begin{lemma}
	\label{containment multi}
	If $\phi \in \LL_c^n(A:\ell)$ then $\phi \in \LL_w^n(A)$.
	That is, $\LL_c^n(A:\ell) = \LL_w^n(A:\ell) \subseteq \LL_w^n(A)$.
	\end{lemma}
	
	\begin{proof}
	Since $T_n$ takes $\ell$-multimodular maps to $\cl{\pi(\ell)}$-multimodular maps we have
	\begin{center}
	\begin{tabular}{l l l}
	$S_nT_n\phi(a_1, \dots, a_n)$  & $ = T_n\phi(\theta^{-1}(a_1), \dots, \theta^{-1}(a_n))$ & by the definition of $S_n$ \\
							& $=T_n\phi(p\pi(a_1), \dots, p\pi(a_n))$ & by the properties of $\theta$ \\
							& $= p \cdot T_n\phi(\pi(a_1), \dots, \pi(a_n))$  & since $p$ is central and\\
								&								&  $T_n\phi$ is $\cl{\pi(\ell)}$-multimodular. \\
								& $= p \cdot \phi(a_1, \dots, a_n)$		&by the definition of $T_n$. \\
								& $= \phi(a_1, \dots, a_n)$		&by line (\ref{action}).

	\end{tabular}
	\end{center}
	
	Then, since $S_nT_n$ is a projection from $\LL_c^n(A)$ onto $\LL_w^n(A,\BB)$, we are done.
	\end{proof}
	
	\begin{lemma}
	\label{up one}
	If $H_c^{n-1}(A) = 0$ then the map $H_w^n(A) \to H_c^n(A)$ is an injection.
	\end{lemma}
	
	\begin{proof}
	Suppose that $\phi \in \LL_w^n(A)$ with $\phi = \partial \psi$ for some $\psi \in \LL_c^{n-1}(\A)$.
	So if $H_c^{n-1}(A) = 0$, we have $H_c^{n-1}(A:\ell) \iso H_c^{n-1}(A)$, since by Remark \ref{key} $H_c^{n-1}(A:\ell) = 0$.
	Hence, without loss of generality, we may assume $\psi \in \LL_c^{n-1}(A:\ell) \subseteq \LL_w^{n-1}(\A)$. 
	Thus, $[\phi] = 0$ in $H_w^n(A)$ and we are done.
	\end{proof}

	We are now ready to prove Theorem \ref{main 3}.

	 \begin{proof}[Proof of Theorem \ref{main 3}]
	 Since derivations are automatically weakly continuous (\cite{10.2307/1970433} Lemma 3), $H_w^1(A) = H_c^1(A) = 0$. 
	 Next, given any $n>1$ we have that $H_c^{n-1}(A) = 0$, so by Lemma \ref{up one}, $H_w^n(A) \to H_c^n(A)$ is an injection.
	 Moreover, $H_c^n(A) = 0$, so we must have that $H_w^n(A) = 0$.
	 Since $n$ was arbitrary we are done.
	 \end{proof}


\bibliographystyle{amsplain}
\renewcommand\refname{\centering Bibliography}
\bibliography{/Users/mattlorentz/Dropbox/research/preambles_and_biblio/biblio.bib} 

\providecommand{\bysame}{\leavevmode\hbox to3em{\hrulefill}\thinspace}
\providecommand{\MR}{\relax\ifhmode\unskip\space\fi MR }
\providecommand{\MRhref}[2]{%
  \href{http://www.ams.org/mathscinet-getitem?mr=#1}{#2}
}
\providecommand{\href}[2]{#2}
\begin{thebibliography}{1}

\bibitem{blackadar2006operator}
B.~Blackadar, \emph{Operator algebras}, Encyclopaedia of Mathematical Sciences,
  vol. 122, Springer-Verlag, Berlin, 2006, Theory of $C^*$-algebras and von
  Neumann algebras, Operator Algebras and Non-commutative Geometry, III.
  \MR{2188261}

\bibitem{braga2018rigidity}
B.~M. Braga and I.~Farah, \emph{On the rigidity of uniform {R}oe algebras over
  uniformly locally finite coarse spaces}, Trans. Amer. Math. Soc. \textbf{374}
  (2021), no.~2, 1007--1040. \MR{4196385}

\bibitem{hochschild1945cohomology}
G.~Hochschild, \emph{On the cohomology groups of an associative algebra}, Ann.
  of Math. (2) \textbf{46} (1945), 58--67. \MR{11076}

\bibitem{10.2307/1970433}
R.~Kadison, \emph{Derivations of operator algebras}, Ann. of Math. (2)
  \textbf{83} (1966), 280--293. \MR{193527}

\bibitem{MR318907}
R.~Kadison and J.~Ringrose, \emph{Cohomology of operator algebras. {II}.
  {E}xtended cobounding and the hyperfinite case}, Ark. Mat. \textbf{9} (1971),
  55--63. \MR{318907}

\bibitem{lorentz2020}
M.~Lorentz and R.~Willett, \emph{Bounded derivations on uniform {R}oe
  algebras}, Rocky Mountain J. Math. \textbf{50} (2020), no.~5, 1747--1758.
  \MR{4170683}

\bibitem{reed2012methods}
M.~Reed and B.~Simon, \emph{Methods of modern mathematical physics. {I}.
  {F}unctional analysis}, Academic Press, New York-London, 1972. \MR{0493419}

\bibitem{sinclair1995hochschild}
A.~Sinclair and R.~Smith, \emph{Hochschild cohomology of von {N}eumann
  algebras}, London Mathematical Society Lecture Note Series, vol. 203,
  Cambridge University Press, Cambridge, 1995. \MR{1336825}

\bibitem{takesaki2013theory}
M.~Takesaki, \emph{Theory of operator algebras. {I}}, Encyclopaedia of
  Mathematical Sciences, vol. 124, Springer-Verlag, Berlin, 2002, Reprint of
  the first (1979) edition, Operator Algebras and Non-commutative Geometry, 5.
  \MR{1873025}

\end{thebibliography}
\addcontentsline{toc}{section}{\numberline{}Bibliography}	
\end{document}